\documentclass[11pt]{amsart}
\usepackage{amscd}
\usepackage{url,enumerate,color}
\usepackage[cmtip, all]{xy}

\usepackage[top=30mm,right=30mm,bottom=30mm,left=20mm]{geometry}

\usepackage{amsfonts}
\usepackage{enumitem}
\usepackage{amsmath}
\usepackage{array}
\usepackage{amssymb}
\usepackage{amsthm}
\usepackage{hyperref}
\usepackage{mathrsfs}
\usepackage{cite}
\usepackage{aliascnt}
\usepackage{centernot}
\usepackage{dsfont}
\usepackage{stmaryrd}
\usepackage{tikz-cd}
\DeclareMathOperator{\Irr}{Irr}
\DeclareMathOperator{\GL}{GL}

\DeclareMathOperator{\SU}{SU}

\DeclareMathOperator{\SL}{SL}
\DeclareMathOperator{\PSL}{PSL}

\DeclareMathOperator{\PGL}{PGL}
\DeclareMathOperator{\Aut}{Aut}
\DeclareMathOperator{\Char}{\mathsf{Char}}
\DeclareMathOperator{\inv}{\textsf{inv}}
\DeclareMathOperator{\geom}{{geom}}
\DeclareMathOperator{\arith}{{arith}}

\DeclareMathOperator{\Trace}{Trace}
\DeclareMathOperator{\Norm}{Norm}

\DeclareMathOperator{\lcm}{lcm}

\DeclareMathOperator{\Wild}{\mathsf{Wild}}
\DeclareMathOperator{\Tame}{\mathsf{Tame}}
\DeclareMathOperator{\diag}{diag}

\DeclareMathOperator{\Weil}{\mathsf{Weil}}

\begin{document}
\numberwithin{equation}{section}
\setlength{\parindent}{15pt}

\newtheorem*{dfn}{Definition}
\newtheorem*{qst}{Question}
\newtheorem{thm}{Theorem}[section]
\newcommand{\thmautorefname}{Theorem}

\newtheorem*{thm*}{Theorem}

\newaliascnt{prop}{thm}
\newtheorem{prop}[prop]{Proposition}
\aliascntresetthe{prop}
\newcommand{\propautorefname}{Proposition}

\newaliascnt{lem}{thm}
\newtheorem{lem}[lem]{Lemma}
\aliascntresetthe{lem}
\newcommand{\lemautorefname}{Lemma}

\newaliascnt{crl}{thm}
\newtheorem{crl}[crl]{Corollary}
\aliascntresetthe{crl}
\newcommand{\crlautorefname}{Corollary}
\theoremstyle{definition}
\newaliascnt{rmk}{thm}
\newtheorem{rmk}[rmk]{Remark}
\aliascntresetthe{rmk}
\newcommand{\rmkautorefname}{Remark}

\title{Hypergeometric Sheaves and Finite General Linear Groups}
\author{Lee Tae Young}
\address{Department of Mathematics, Rutgers University, Piscataway, NJ 08854, USA}
\email{tl639@math.rutgers.edu}

\maketitle

\begin{abstract}
We find all irreducible hypergeometric sheaves whose geometric monodromy group is finite, almost quasisimple and has the projective special linear group $\PSL_n(q)$ with $n\geq 3$ as a composition factor. We use the classification of semisimple elements with specific spectra in irreducible Weil representations to prove that if an irreducible hypergeometric sheaf has such geometric monodromy group, then it must be of certain form. Then we extend results of Katz and Tiep on a prototypical family of such sheaves to full generality to show that these hypergeometric sheaves do have such geometric monodromy groups, and that they have some connection to a construction of Abhyankar. 
\end{abstract}

\textbf{2010 Mathematics Subject Classification:} Primary 11T23, Secondary 20C15, 20C33, 20D06

\textbf{Keywords:} Finite General Linear Groups, Monodromy Groups, Hypergeometric Sheaves, Weil Representations.

\tableofcontents

\subsection*{Introduction}

Let $p$ be a prime and let $q$ be a power of $p$. A conjecture of Abhyankar \cite{Ab57}, proved by Harbater \cite{Har}, says that the finite quotient groups of the \'etale fundamental group of the multiplicative group $\mathbb{G}_m:= \mathbb{A}^1\setminus\{0\}$ over $\overline{\mathbb{F}_p}$ are precisely the finite groups $G$ such that $G/\mathbf{O}^{p'}(G)$ is cyclic, where $\mathbf{O}^{p'}(G)$ is the subgroup generated by all Sylow $p$-subgroups of $G$. Since this proof is nonconstructive, one might want to realize each of these groups in a faithful complex representation as the monodromy group of an explicitly written $\overline{\mathbb{Q}_\ell}$-local system over $\mathbb{G}_m$.

Hypergeometric sheaves are the simplest rigid local systems over $\mathbb{G}_m$, in the sense that they have the sum of Swan conductors equal to $1$, which is the lowest possible nonzero value. Katz, Rojas-Leon and Tiep \cite{KT1}, \cite{KT3}, \cite{KRLT4} used hypergeometric sheaves to realize many quotient groups of $\pi_1^{et}(\mathbb{G}_m/\mathbb{F}_q)$. In \cite{KT2}, Katz and Tiep studied the converse direction: they showed that if the geometric monodromy group of a hypergeometric sheaves satisfying a mild condition $(\mathbf{S}+)$ is finite, then it is either almost quasisimple or an ``extraspecial normalizer''. For the almost quasisimple case, they also gave a list of all possible pairs $(S,V)$ of finite simple groups $S$ and their complex representations $V$ which can occur as the unique nonabelian simple factor of the geometric monodromy group of irreducible hypergeometric sheaves, cf. \cite[Section 10]{KT2}.

In this paper, we return to the original viewpoint of constructing local systems that realize given finite groups as their monodromy groups. We focus on one of the five ``generic'' families of finite almost quasisimple groups listed in \cite[Section 10]{KT2}, namely those with the unique nonabelian composition factor $\PSL_n(\mathbb{F}_q)$ with $n\geq 3$. Katz and Tiep \cite{KT1} found some examples of hypergeometric sheaves realizing some of these groups. However, whether these are the only hypergeometric sheaves realizing these groups, and the analogous statements for other families of finite almost quasisimple groups, were not known. Here we give the first result in this direction: we give a complete list of irreducible hypergeometric sheaves for groups ``coming from'' $\PSL_n(\mathbb{F}_q)$. 

The main results of this paper, namely \autoref{large_rank}, \autoref{small_rank} and \autoref{Ggeom_of_constituents}, says that for $n\geq 3$, the irreducible hypergeometric sheaves whose geometric monodromy group is finite almost quasisimple having $\PSL_n(\mathbb{F}_q)$ as the unique nonabelian composition factor are precisely those of the form 
\begin{align*}
\mathcal{H}yp_\psi(\Char(\frac{q^n-1}{q-1},\chi^{(b+c)j}); \Char(\frac{q^m-1}{q-1}, \chi^{bj}) \cup \Char(\frac{q^{n-m}-1}{q-1}, \chi^{cj}))\otimes \mathcal{L}_\varphi
\end{align*}
or
\begin{align*}
 \mathcal{H}yp_\psi(\Char(\frac{q^n-1}{q-1})\setminus\{\mathds{1}\}; \Char(\frac{q^m-1}{q-1}) \cup (\Char(\frac{q^{n-m}-1}{q-1})\setminus\{\mathds{1}\}))\otimes \mathcal{L}_\varphi
\end{align*}
where $\varphi$ is a multiplicative character of $\overline{\mathbb{F}_p}$ of finite order, $\chi$ is that of order precisely $q-1$, $\psi$ is a nontrivial additive character of $\mathbb{F}_p$, and $n, b, c$ are integers satisfying certain conditions. 

In section 1, we fix notations and review some known facts about irreducible hypergeometric sheaves and their monodromy groups, which will be needed in the subsequent sections. In section 2, we set up some notations regarding the Weil representations of $\GL_n(\mathbb{F}_q)$, and make some observations which will be used in subsequent sections. Section 3 studies the action of certain $p$-subgroups of $\GL_n(\mathbb{F}_q)$ on the irreducible Weil modules. Together with some facts we review in section 1, the results of section 2 and 3 completely determine the possible local monodromies at $0$ and $\infty$ of the hypergeometric sheaves we want to study. However, the local pictures at these two points are studied separately, so we need to determine which pairs of them can actually arise as local monodromy of a hypergeometric sheaf with the desired geometric monodromy group. In section 4, we use two new ideas to achieve this: the use of determinant sheaves of irreducible hypergeometric sheaves to find a connection between the local monodromies at $0$ and $\infty$, and some new techniques to find counterexamples for certain inequality called the ``$V$-test''. Thanks to these, we obtain short lists \autoref{large_rank} and \autoref{small_rank} of candidate hypergeometric sheaves. In section 5, we prove that these sheaves do have the desired geometric monodromy groups. This is done by extending the method used in \cite{KT1} to study a smaller family of hypergeometric sheaves. This family in \cite{KT1} also had some connection to a work of Abhyankar \cite{Ab94}. We briefly discuss a generalization of this connection to the sheaves in \autoref{large_rank} and \autoref{small_rank} in the final section.

\subsection*{Acknowledgments}
I would like to thank my Ph. D. advisor, Professor Pham Huu Tiep, for raising this problem and his devoted support and guidance. I would also thank Professor Nicholas Katz for helpful discussions. I gratefully acknowledge the support of NSF (grants DMS-1840702 and DMS-2200850). 

\section{Preliminary Results and the Basic Set-up}

Let $\overline{\mathbb{F}_p}$ be the algebraic closure of the finite field of characteristic $p$. Fix another prime $\ell\neq p$, and let $\overline{\mathbb{Q}}_\ell$ be the algebraic closure of the field of $\ell$-adic numbers. Throughout this paper, we will not distinguish between the $\mathbb{C}$-representations and $\overline{\mathbb{Q}_\ell}$-representations of finite groups. Let $K$ be a finite subfield of $\overline{\mathbb{F}_p}$. We will understand lisse $\overline{\mathbb{Q}_\ell}$-local systems over $\mathbb{G}_m/K$ as continuous representations of the \'etale fundamental group $\pi_1^{et}(\mathbb{G}_m/K)$. The Zariski closures of the image of $\pi_1^{et}(\mathbb{G}_m/K)$ and the subgroup $\pi_1^{\geom}(\mathbb{G}_m) :=\pi_1^{et}(\mathbb{G}_m/\overline{\mathbb{F}_p})$ under this representation are called the \emph{arithmetic} and \emph{geometric monodromy group}, often denoted by $G_{\arith}$ and $G_{\geom}$, respectively, of this sheaf. 

To study the monodromy of lisse $\overline{\mathbb{Q}_\ell}$-sheaves over $\mathbb{G}_m$, we need to look at their local monodromy at $0$ and $\infty$, that is, the restrictions to the inertia subgroups $I(0)$ and $I(\infty)$ of $\pi_1^{et}(\mathbb{G}_m)$. The inertia subgroup $I(0)$ has a normal pro-$p$-subgroup, namely the wild inertia subgroup $P(0)$. The quotient $I(0)/P(0)$ is a pro-cyclic group of pro-order prime to $p$. Fix an element $\gamma_0$ of $I(0)$ of pro-order prime to $p$, such that $\gamma_0 P(0)$ is a topological generator of $I(0)/P(0)$. Similarly, fix $\gamma_\infty\in I(\infty)$ of pro-order prime to $p$ such that $\gamma_\infty P(\infty)$ is a topological generator of $I(\infty)/P(\infty)$. 

Let $\psi: \mathbb{F}_p \rightarrow  \overline{\mathbb{Q}_\ell}^\times$ be a nontrivial additive character. Given $D$ multiplicative characters $\chi_1,\dots,\chi_D: K^\times \rightarrow \overline{\mathbb{Q}_\ell}^\times$ and $m$ multiplicative characters $\rho_1,\dots, \rho_m: K^\times \rightarrow \overline{\mathbb{Q}_\ell}^\times$, where $D>m\geq 0$, one can define the \emph{hypergeometric sheaf of type $(D,m)$}: $$\mathcal{H}yp_\psi(\chi_1,\dots,\chi_D; \rho_1,\dots,\rho_m).$$ We will assume that $\{\chi_1,\dots,\chi_D\}\cap \{\rho_1,\dots,\rho_m\}=\emptyset$. Under this assumption, this is lisse on $\mathbb{G}_m/K$, geometrically irreducible, has rank $D$, and pure of weight $D+m-1$. For the details of these facts and other basic theory of hypergeometric sheaves, see \cite[Chapter 8]{K-ESDE}.

The characters $\chi_1,\dots, \chi_D$ are called the ``upstairs characters'' of this hypergeometric sheaf. The local monodromy at $0$ of this sheaf is tame, and given by the direct sum of Jordan blocks of the Kummer sheaves $\mathcal{L}_{\chi_i}$ defined by the upstairs characters. In particular, the characters $\chi_1,\dots, \chi_D$ must be pairwise distinct whenever the geometric monodromy group is finite, and in this case the local monodromy is just the direct sum $\bigoplus_{i=1}^{D}\mathcal{L}_{\chi_i}$. Since it is tame, we can view it as a continuous representation of $I(0)/P(0) = \langle \gamma_0P(0)\rangle$. In particular, the image of $\gamma_0$ will completely determine this representation (up to isomorphism). 

The ``downstairs characters'' $\rho_1,\dots, \rho_m$ have a similar but slightly different property. The hypergeometric sheaf $\mathcal{H}yp_\psi(\chi_1,\dots,\chi_D; \rho_1,\dots,\rho_m)$ is not tame at $\infty$, so the local monodromy at $\infty$ can be written as a direct sum $\Tame \oplus \Wild$. Here, $\Tame$ is a tame representation of rank $m$ determined by the downstairs characters in the same way as how the upstairs characters determine the local monodromy at $0$. In addition to that, we have a totally wild part $\Wild$ of dimension $D-m$ and Swan conductor $1$. Hence, to determine the downstairs character, looking at the image of $\gamma_\infty$ alone is insufficient; we should also look at the image of $P(\infty)$ and use the following result.
\begin{prop}\label{wild_basic_properties}{\upshape{\cite[Proposition 4.10]{KT2}\cite[Proposition 5.9]{KRLT4}}}
Let $\mathcal{H}$ be an irreducible hypergeometric sheaf of type $(D,m)$ with $D>m\geq 0$. If $D-m = p^a W_0$ for some integer $a\geq 0$ and $p\nmid W_0$, then we have the following:
\begin{enumerate}[label={\upshape{(\roman*)}}]
\item $\Wild|_{P(\infty)}$ is a direct sum of $W_0$ multiplicative translates of $P|_{P(\infty)}$ by $\mu_{W_0}$, where $P$ is an irreducible $I(\infty)$-representation of dimension $p^a$ and Swan conductor $1$.
\item $\gamma_\infty$ cyclically permutes these $W_0$ irreducible constituents of $\Wild|_{P(\infty)}$.
\item If $a=0$, then the image of $P(\infty)$ is isomorphic to the additive group of the finite field $\mathbb{F}_p(\mu_{D-m})$.
\item If $a>0$, then there exists a root of unity $\zeta$ whose order is prime to $p$, such that the spectrum of $\gamma_\infty^{W_0}$ on each irreducible constituents of $\Wild|_{P(\infty)}$ is $\zeta\cdot (\mu_{p^a+1}\setminus \{1\})$. 
\end{enumerate}
\end{prop}

If an irreducible hypergeometric sheaf is not primitive, then it is either Kummer induced or Belyi induced, and both cases can be easily recognized from the upstairs and downstairs characters, cf. \cite[Proposition 1.2]{KRLT2}. If we restrict ourselves to the primitive cases, then \cite[Theorem 5.2.9]{KT4} and \cite[Lemma 1.1]{KT2} tells us that if our hypergeometric sheaf is of type $(D,m)$ with $D>m$ and $D\neq 4,8,9$, and if the geometric monodromy group $G_{\geom}$ of this sheaf is finite, then $G_{\geom}$ is either almost quasisimple or an ``extraspecial normalizer''. Moreover, if $D$ is not a prime power and not $1$, then $G_{\geom}$ is almost quasisimple, and the unique nonabelian composition factor and its representation are as listed in \cite[Section 10]{KT2}.

In this paper, we study the sheaves which realize \cite[Section 10, case (b)]{KT2}: the irreducible (but not necessarily primitive) hypergeometric sheaves whose geometric monodromy groups are finite, almost quasisimple with unique nonabelian composition factor $\PSL_n(\mathbb{F}_q)$, $n\geq 3$. Such hypergeometric sheaves are known to exist, and one construction can be found in \cite{KT1}. Most of such hypergeometric sheaves are known to satisfy several nice properties. We will need some of these properties, which we list below for convenience.
\begin{prop}\label{basic_facts}{\upshape{\cite[Proposition 8.15.2]{K-ESDE}\cite[Theorem 6.6(ii), Theorem 7.3, Theorem 8.1, Corollary 8.4]{KT2}}}
Let $\mathcal{H}yp_\psi(\chi_1,\dots,\chi_D; \rho_1,\dots,\rho_m)$ be an irreducible hypergeometric sheaf over $\mathbb{G}_m/K$ for some finite field $K$. Suppose that the geometric monodromy group $G$ is finite, almost quasisimple with unique nonabelian composition factor $\PSL_n(\mathbb{F}_q)$ for an integer $n\geq 3$ and a power $q$ of a prime $p$. Then:
\begin{enumerate}[label={\upshape{(\alph*)}}]
\item The upstairs and downstairs characters are pairwise distinct.
\item The characteristic of $K$ is $p$, unless $(n,q)$ is one of $(3,2), (4,2), (3,3), (3,4)$ and $D\leq 22$.
\item Suppose that $(n,q)\neq (3,4)$. Then the quasisimple layer $E(G)$ of $G$ is a quotient of $\SL_n(\mathbb{F}_q)$. Moreover, the monodromy representation as a representation of $E(G)$ is an irreducible Weil representation of this quotient of $\SL_n(\mathbb{F}_q)$.
\item If $(n,q)\neq (3,2), (3,3), (3,4)$, then the image $\overline{g_0}$ of $\gamma_0$ under the map $\pi_1^{et}(\mathbb{G}_m) \rightarrow  G \rightarrow  G/\mathbf{Z}(G) \subseteq \Aut(\PSL_n(\mathbb{F}_q))$ lies in $\PGL_n(\mathbb{F}_q)$. If in addition $D-m\geq 2$, then $G/\mathbf{Z}(G)\cong \PGL_n(\mathbb{F}_q)$.
\end{enumerate}
\end{prop}

When $D-m=1$, \cite[Corollary 8.4]{KT2} (which is the second part of \autoref{basic_facts}(d)) does not apply, since it relies on \cite[Theorem 4.1]{KT2} which requires $D-m\geq 2$. However, we can at least prove the following weaker version.

\begin{prop}\label{D-m=1}
Let $\mathcal{H}$ be a hypergeometric sheaf as in \autoref{basic_facts}, and let $G=G_{\geom}$ be the geometric monodromy group of this. Assume that the conclusion of \autoref{basic_facts}(c) holds. Suppose that $D-m=1$ and $(n,q)\neq (3,2), (3,3),(3,4)$. Let $g_\infty\in G$ be the image of $\gamma_\infty\in I(\infty)$ in $G$, and let $\overline{g_\infty}\in G/\mathbf{Z}(G)$ be its image. Then $\overline{g_\infty}\in \PGL_n(\mathbb{F}_q)$.
\end{prop}

\begin{proof}
The spectrum of $g_\infty$ on $\mathcal{H}$ is the union of the spectrum on $\Tame$ and that on $\Wild$. The spectrum on $\Tame$ corresponds to the upstairs characters, so in particular $g_\infty$ has at least $m = D-1$ distinct eigenvalues. Since the restriction of the monodromy representation is an irreducible Weil representation of $\SL_n(\mathbb{F}_q)$, $D$ is either $\frac{q^n-q}{q-1}$ or $\frac{q^n-1}{q-1}$. Therefore, the order of $g_\infty$ is at least $\frac{q^n-1}{q-1}-2$. Now we can apply the first part of the proof of \cite[Theorem 8.1]{KT2}. 
\end{proof}

Instead of excluding all the exceptional pairs of $(n,q)$ in \autoref{basic_facts} and \autoref{D-m=1}, we want to include those hypergeometric sheaves which satisfies the conclusions of the above propositions. Therefore, we will study the hypergeometric sheaves $\mathcal{H}$ with the geometric monodromy group $G$ which satisfies the following conditions:
\begin{equation}
\left(\begin{array}{l}
\mathcal{H}\text{ is irreducible and lisse on }\mathbb{G}_m/K\text{ for a finite extension }K/\mathbb{F}_p.\\G\text{ is finite, almost quasisimple with unique composition factor } \PSL_n(\mathbb{F}_q)\text{ for some}\\\text{integer }n\geq 3\text{ and a power }q\text{ of }p.\\\text{The images of }\gamma_0,\gamma_\infty\text{ under the map }\pi_1^{et}(\mathbb{G}_m) \rightarrow  G \rightarrow  G/\mathbf{Z}(G)\text{ are in }\PGL_n(\mathbb{F}_q).\\\text{The quasisimple layer }E(G)\text{ is a quotient of }\SL_n(\mathbb{F}_q)\text{, and the restriction of the}\\\text{monodromy representation of }\mathcal{H}\text{ to }E(G)\text{ is an irreducible Weil representation.}
\end{array} \right)\label{hyp_condition}\tag{$\star$}
\end{equation}
By \autoref{basic_facts} and \autoref{D-m=1}, except when $(n,q)=(3,2),(3,3),(3,4),(4,2)$, the last two conditions in \eqref{hyp_condition} and that $K$ has characteristic $p$ are redundant, and \eqref{hyp_condition} is equivalent to:
\begin{equation*}
\left(\begin{array}{l}
\mathcal{H}\text{ is irreducible and lisse on }\mathbb{G}_m/K\text{ for a finite field }K.\\G\text{ is finite, almost quasisimple with unique composition factor }\PSL_n(\mathbb{F}_q)\text{ for some}\\\text{integer }n\geq 3\text{ and a power }q\text{ of }p.
\end{array} \right)
\end{equation*}

\begin{rmk}\label{remark_sec1}
In the situation of \eqref{hyp_condition}, let $g_0, g_\infty$ be the images of $\gamma_0,\gamma_\infty$ in $G$ under the monodromy representation, and let $\overline{g_0}, \overline{g_\infty}$ be their images in $G/\mathbf{Z}(G)\leq \Aut(\PSL_n(\mathbb{F}_q))$, so that $\overline{g_0}, \overline{g_\infty}\in \PGL_n(\mathbb{F}_q)$. The monodromy representation gives a projective representation of $\PSL_n(\mathbb{F}_q)=E(G)/\mathbf{Z}(E(G))$, which comes from an irreducible Weil representation of $\SL_n(\mathbb{F}_q)$. Hence, if we take the restriction of the monodromy representation to the subgroup $\langle E(G),g_0, g_\infty\rangle$ of $G$, then this gives a projective representation of the subgroup $\langle E(G)/\mathbf{Z}(E(G)),\overline{g_0},\overline{g_\infty}\rangle \leq \PGL_n(\mathbb{F}_q)$. This can be lifted to an irreducible Weil representation of the corresponding subgroup $H$ of $\GL_n(\mathbb{F}_q)$. Therefore, the spectrum of $\gamma_0$ on $\mathcal{H}$ is just a root of unity times the spectrum of an inverse image $h_0\in H\leq \GL_n(\mathbb{F}_q)$ of $\overline{g_0}$ on this irreducible Weil representation. The same statement holds for $\gamma_\infty$ with another root of unity. 
\end{rmk}

\section{Weil Representations of Finite General Linear Groups}

Suppose that we have a hypergeometric sheaf $\mathcal{H}$ with the geometric monodromy group $G$ which satisfies \eqref{hyp_condition}. The discussions in section 1, including \autoref{wild_basic_properties} and \autoref{basic_facts}, tells us that the spectrum of $\gamma_0$ on $\mathcal{H}$ cannot have an eigenvalue with multiplicity larger than $1$, and the spectrum of $\gamma_\infty$ on each of $\Tame$ and $\Wild$ also have the same property. As we saw in \autoref{remark_sec1}, these spectra is just a root of unity times the spectra of some elements of $\GL_n(\mathbb{F}_q)$ on some irreducible Weil representation. In this section, we recall the definition of irreducible Weil representation, and fix some notations which will be used in later sections.

Fix a generator $\alpha$ of the cyclic group $\mathbb{F}_q^\times$, and a primitive $(q-1)$th root of unity $\lambda\in \overline{\mathbb{Q}_\ell}^\times$. Let $\eta$ be a multiplicative character of $\mathbb{F}_q$ that maps $\alpha$ to $\lambda$. Consider the natural (left) permutation action of $\GL_n(\mathbb{F}_q)$ on $\mathbb{F}_q^n$. The corresponding $\mathbb{C}\GL_n(\mathbb{F}_q)$-module is called the total Weil module. We will denote this by $\Weil$. It has a standard basis $\{e_v \mid v\in \mathbb{F}_q^n\}$, and each $g\in \GL_n(\mathbb{F}_q)$ acts by $e_v\mapsto e_{gv}$. For each $0\neq v\in \mathbb{F}_q^n$ and each $j=0,\dots, q-2$, define $v^{(j)} = \sum_{i=0}^{q-2}\lambda^{-ij}e_{\alpha^i v}$. Then $$(\alpha I).v^{(j)} = \sum_{i=0}^{q-2}\lambda^{-ij}(\alpha I).e_{\alpha^i v} = \sum_{i=0}^{q-2}\lambda^{-ij}e_{\alpha^{i+1}v} = \sum_{i=0}^{q-2}\lambda^{-(i-1)j}e_{\alpha^i v} = \lambda^j v^{(j)}.$$ In particular, the element $\alpha I\in \mathbf{Z}(\GL_n(\mathbb{F}_q))$ has an eigenvalue $\lambda^j$ on $\Weil$, and $v^{(j)}$ is an eigenvector associated to this eigenvalue. Note that $(\alpha v)^{(j)} = \lambda^j v^{(j)}$, so if we choose one nonzero $v$ from each one-dimensional subspace of $\mathbb{F}_q^n$, then the $v^{(j)}$'s for those $v$ together with $e_0$ form a basis of $\Weil$. Let $\Weil_j = \mathbb{C}\langle v^{(j)}\mid 0\neq v\in \mathbb{F}_q^n\rangle$. Then we get the direct sum decomposition $$\Weil = \mathbb{C}e_0 \oplus \bigoplus_{j=0}^{q-1}\Weil_j.$$ The submodules $\Weil_j$ have dimension $\frac{q^n-1}{q-1}$, and the restriction of $\Weil_j$ to $\mathbf{Z}(\GL_n(\mathbb{F}_q))$ is precisely $\frac{q^n-1}{q-1}\eta^j$; recall that $\eta$ is the linear character which maps $\alpha$ to $\lambda$. Moreover, $\Weil_j$ is irreducible unless $j=0$, in which case $\Weil_0 = \mathds{1} \oplus \Weil_0'$. The $\mathbb{C}\GL_n(\mathbb{F}_q)$-modules $\Weil_j$ and $\Weil_0'$, together with their tensor products with linear characters of $\GL_n(\mathbb{F}_q)$, are called the irreducible Weil modules. 

Fix an element $g$ of $\GL_n(\mathbb{F}_q)$. To study the spectrum of the action of $g$ on $\Weil_j$, it is convenient to make the following definitions. For a nonzero vector $v\in \mathbb{F}_q^n\setminus \{0\}$, let $s_v$ be the smallest positive integer which satisfies $g^{s_v}v = \alpha^{t}v$ for some integer $t$, and let $t_v$ be this $t$. For two nonzero vectors $v, w\in \mathbb{F}_q^n\setminus \{0\}$, we will say $v\sim_g w$ if $w = \beta g^rv$ for some $\beta\in \mathbb{F}_q^\times$ and some nonnegative integer $r$, that is, if $v$ and $w$ lie in the same $\langle g,\mathbf{Z}(\GL_n(\mathbb{F}_q))\rangle$-orbit. This defines an equivalence relation on $\mathbb{F}_q^n\setminus\{0\}$. Also, let $V(g;v)$ and $\Weil_j(g;v)$ denote the $g$-cyclic subspaces of $\mathbb{F}_q^n$ and $\Weil_j$ generated by $v$ and $v^{(j)}$, respectively. Then we can easily see the following properties.

\begin{lem}\label{equiv_relation}
Let $v,w\in \mathbb{F}_q^n\setminus \{0\}$.
\begin{enumerate}[label={\upshape{(\arabic*)}}]
\item $v\sim_g w$ if and only if $\Weil_j(g;v)=\Weil_j(g;w)$. Moreover, if $v\not\sim_g w$, then $\Weil_j(g;v)\cap \Weil_j(g;w)=0$.
\item $\Weil_j$ is the direct sum of subspaces of the form $\Weil_j(g;v)$, one for each equivalence class of $\sim_g$.
\item $s_v$ and $t_v$ only depends on the $\sim_g$-equivalence class of $v$.
\item The eigenvalues of the action of $g$ on $\Weil_j(g;v)$ are the $s_v$th roots of $\lambda^{t_vj}$.
\end{enumerate}
\end{lem}

\begin{proof}
For (4), let $\xi\in \mathbb{C}$ be a number such that $\xi^{s_v} = \lambda^t$. The vector $\sum_{i=0}^{s_v-1}\xi^{-ij}g^i.v^{(j)}\in \Weil_j(g;v)$ then satisfies 
\begin{align*}
g.\left(\sum_{i=0}^{s_v-1}\xi^{-ij}g^i.v^{(j)}\right) &= \sum_{i=0}^{s_v-1}\xi^{-ij}g^{i+1}.v^{(j)} = \sum_{i=1}^{s_v-1}\xi^{-(i-1)j}g^{i}.v^{(j)} + \xi^{-(s_v-1)j}g^{s_v}.v^{(j)} \\&= \sum_{i=1}^{s_v-1}\xi^{-(i-1)j}g^{i}.v^{(j)} + \xi^j\left(\sum_{k=0}^{q-2}\lambda^{-(k+t)j}e_{\alpha^{k+t}v}\right) \\&=  \sum_{i=1}^{s_v-1}\xi^{-(i-1)j}g^{i}.v^{(j)} + \xi^jv^{(j)} = \xi^{j}\sum_{i=0}^{s_v-1}\xi^{-ij}g^{i}.v^{(j)}.
\end{align*}
Therefore, this vector is an eigenvector of the action of $g$ on $\Weil_j(g;v)$ with eigenvalue $\xi^j$. Since we can choose $s_v$ distinct $\xi$'s and $\dim \Weil_j(g;v)=s_v$, these vectors form an eigenbasis of the action of $g$ on this cyclic subspace.

\end{proof}

Since $\gamma_0$ and $\gamma_\infty$ has pro-order prime to $p$, to find the possible spectra of $\gamma_0$ and $\gamma_\infty$ on $\mathcal{H}$ (up to multiplication by roots of unity), it is sufficient to use the following result.

\begin{thm}{\upshape{\cite[Proposition 10.3.6]{KT4}}}\label{m2sp_elts}
Suppose that the action of a $p'$-element $g\in \GL_n(\mathbb{F}_q)$ on $\Weil_j$ for some $j\in \{0,\dots,q-2\}$ has no eigenvalue of multiplicity larger than $2$. Then $g$ is one of the following.
\begin{enumerate}[label={\upshape{(\alph*)}}]
\item $g=\alpha_n^a$, where $\alpha_n$ is a generator of $\mathbb{F}_{q^n}^\times$ such that $\alpha_n^{\frac{q^n-1}{q-1}}=\alpha$, viewed as an element of $\GL_n(\mathbb{F}_q)$ via some embedding $\GL_1(\mathbb{F}_{q^n}) \hookrightarrow \GL_n(\mathbb{F}_q)$, and $a$ is an integer relatively prime to $|\alpha_n|/(q-1) = (q^n-1)/(q-1)$. The spectrum of the action of $g$ on $\Weil_j$ is the $\frac{q^n-1}{q-1}$th roots of $\lambda^{aj}$.
\item The squares of the elements described in (a), when $(q^n-1)/(q-1)$ is even.
\item $g = \alpha_m^b \oplus \alpha_{n-m}^c$. Here, $m$ is a positive integer relatively prime to $n$, and $\alpha_m$ and $\alpha_{n-m}$ are as in (a). $b$ is an integer relatively prime to $\frac{q^m-1}{q-1}$, $c$ is an integer relatively prime to $\frac{q^{n-m}-1}{q-1}$, and $b(n-m)-cm$ must be relatively prime to $q-1$. The spectrum of the action of $g$ on $\Weil_j$ is the $\frac{q^m-1}{q-1}$th roots of $\lambda^{bj}$, the $\frac{q^{n-m}-1}{q-1}$th roots of $\lambda^{cj}$, and the $\frac{(q^m-1)(q^{n-m}-1)}{q-1}$th roots of unity.
\end{enumerate}
\end{thm}

\begin{rmk}
Although the proof of \autoref{m2sp_elts} in \cite{KT4} is short, it relies on deeper results. Here I give a sketch of an alternative proof of this result using \autoref{equiv_relation}. This proof is longer than that of \cite{KT4}, but it has advantages of being elementary and more general, allowing us to also classify the $p$-singular elements with such spectrum on the irreducible Weil representations.

\autoref{equiv_relation} implies that on each $\Weil_j(g;v)$, the action of $g$ has no repeated eigenvalues. Repeated eigenvalues can only occur by appearing in more than one $g$-cyclic subspaces; this is equivalent to saying that there exists $v\not\sim_g w\in \mathbb{F}_q^n\setminus\{0\}$ such that a $s_v$th root of $\lambda^{t_vj}$ is also a $s_w$th root of $\lambda^{t_wj}$. This imposes strong restrictions on the partially ordered set of $g$-cyclic subspaces (ordered by inclusion) if we assume that the action of $g$ has no eigenvalues of multiplicity larger than $2$. In this case, one can easily see that the maximal $g$-cyclic subspaces of $\mathbb{F}_q^n$ must be intersecting trivially, and if we remove $0$ from each of them, then they form a partition of $\mathbb{F}_q^n\setminus\{0\}$. The above observation about $s_v$ and $t_v$ also gives an upper bound for the number of $g$-cyclic subspaces of each dimension. By counting the number of elements in the $g$-cyclic subspaces and using the above observation again, one can show that $\mathbb{F}_q^n$ is itself $g$-cyclic, and there are at most three distinct $\sim_g$-equivalence classes. Using this, one can classify all elements of $\GL_n(\mathbb{F}_q)$ whose action on an irreducible Weil representation has no eigenvalue of multiplicity larger than $2$. 
\end{rmk}

\section{The Action of Wild Inertia Subgroup}

Let $\mathcal{H}$ be an irreducible hypergeometric sheaf with the geometric monodromy group $G=G_{\geom}$ which satisfies \eqref{hyp_condition}. The goal of this section is to completely determine the possible sets of upstairs characters and downstairs characters $\mathcal{H}$ can have. Equivalently, we want to find all possible spectra of the action of $\gamma_0$ on $\mathcal{H}$ and the action of $\gamma_\infty$ on $\Tame$ and $\Wild$. In \autoref{gamma_0_gamma_infty}, we will see that \autoref{m2sp_elts} already gives the answer for $\gamma_0$, but for $\gamma_\infty$, it only gives the possible spectra of the action on $\mathcal{H}=\Tame\oplus \Wild$. To see how these spectra splits into the spectra on $\Tame$ and $\Wild$, we will need to see how the action of $\gamma_\infty$ interacts with the action of $P(\infty)$. 

Let $g_0$ and $g_\infty$ be the images of $\gamma_0$ and $\gamma_\infty$ in $G$. Let $\overline{g_0}$ and $\overline{g_\infty}$ be the images of $g_0$ and $g_\infty$ in $G/\mathbf{Z}(G)$. By \eqref{hyp_condition}, $\overline{g_0},\overline{g_\infty}\in \PGL_n(\mathbb{F}_q)$. We can choose representatives $h_0$ and $h_\infty$ of $\overline{g_0}$ and $\overline{g_\infty}$, respectively, in $\GL_n(\mathbb{F}_q)$. Since $\gamma_0$ and $\gamma_\infty$ have pro-order prime to $p$, both $h_0$ and $h_\infty$ also have order prime to $p$. As explained in \autoref{remark_sec1}, the spectrum of $\gamma_0$ is a root of unity times the spectrum of $h_0$ on an irreducible Weil representation, and the same holds for $\gamma_\infty$ and $h_\infty$.

\begin{prop}\label{gamma_0_gamma_infty}
\begin{enumerate}[label={\upshape{(\arabic*)}}]
\item $h_0$ is an element described in \autoref{m2sp_elts}(a).
\item If $\dim \Wild>1$, then $h_\infty$ is an element described in \autoref{m2sp_elts}(c).
\item If $\dim \Wild=1$, then $h_\infty$ is an element described in \autoref{m2sp_elts}(a).
\end{enumerate}
\end{prop}

\begin{proof}
Since the upstairs characters are pairwise distinct, the action of $\gamma_0$ on $\mathcal{H}$ has simple spectrum. Hence, the action of $h_0$ on $\mathcal{H}$ also has simple spectrum. Among the elements described in \autoref{m2sp_elts}, (a) is the only case with simple spectrum. For $h_\infty$, we know that the downstairs characters are pairwise distinct, so the action on $\Tame$ has simple spectrum. By \autoref{wild_basic_properties}, the action on $\Wild$ also has simple spectrum. Therefore, $h_\infty$ must be as in \autoref{m2sp_elts}. If in addition $\dim \Wild>1$, then we will see in \autoref{R}(c) below that $h_\infty$ stabilizes a nontrivial proper subspace of $\mathbb{F}_q^n$. Only (c) of \autoref{m2sp_elts} has this property.

If $\dim \Wild = 1$, then $h_\infty$ has at most one eigenvalue with multiplicity $2$ on an irreducible Weil representation, and all other eigenvalues have multiplicity $1$. The elements in \autoref{m2sp_elts}(b) and (c) have more than one eigenvalues with multiplicity $2$ on every irreducible Weil representation, so $h_\infty$ must be as in \autoref{m2sp_elts}(a).
\end{proof}

For the rest of this section, assume that $\dim \Wild >1$, so that $G/\mathbf{Z}(G)\cong \PGL_n(\mathbb{F}_q)$ by \autoref{basic_facts}(d). For the exceptional cases $(n,q)=(3,2),(3,3),(3,4)$ where we cannot directly apply \autoref{basic_facts}(d), we still have this property since \autoref{gamma_0_gamma_infty}(1) and the assumption $\dim \Wild>1$ are sufficient to make the proof of \cite[Corollary 8.4]{KT2} valid. Let $J$ and $Q$ be the images of $I(\infty)$ and $P(\infty)$, respectively, in $G$. Let $\overline{J}$ and $\overline{Q}$ be their image in $\PGL_n(\mathbb{F}_q)=G/\mathbf{Z}(G)$. Let $\tilde{J}$ be the preimage of $\overline{J}$ in $\GL_n(\mathbb{F}_q)$, and let $R$ be a Sylow $p$-subgroup of the preimage of $\overline{Q}$ in $\tilde{J}$. Since $P(\infty)$ is a pro-$p$-group, $Q$ and $\overline{Q}$ are $p$-groups.

\begin{lem}\label{R}
\begin{enumerate}[label={\upshape{(\alph*)}}]
\item $R$ is a nontrivial normal Sylow $p$-subgroup of $\tilde{J}$, and $R\cong \overline{Q}$.
\item The subspace $(\mathbb{F}_q^n)^R$ of points fixed by $R$ is nontrivial and proper.
\item There is a nontrivial proper $h_\infty$-stable subspace of $\mathbb{F}_q^n$.
\end{enumerate}
\end{lem}

\begin{proof}
(a) Since $P(\infty)$ is a normal subgroup of $I(\infty)$ such that $I(\infty)/P(\infty)$ has pro-order prime to $p$, $Q$ is a normal Sylow subgroup of $J$. Also, since the wild part of $\mathcal{H}$ is nontrivial, $Q$ cannot be trivial. In fact, $Q\not\leq \mathbf{Z}(G_{\geom})$ by \cite{KRLT4}, so $\overline{Q}$ is a nontrivial normal Sylow subgroup of $\overline{J}$. Since $|\mathbf{Z}(\GL_n(\mathbb{F}_q))|=q-1$ is relatively prime to $p$, it follows that $|\tilde{J}|_p=|\overline{J}|_p=|\overline{Q}|$. 

Let $\tilde{Q}$ be the preimage of $\overline{Q}$ in $\tilde{J}$. Then $|\tilde{Q}| = (q-1)|\overline{Q}|$, so $|R|=|\tilde{Q}|_p=|\overline{Q}|$. Therefore, $R$ is a nontrivial Sylow $p$-subgroup of $\tilde{J}$, and the quotient map $\GL_n(\mathbb{F}_q)\rightarrow \PGL_n(\mathbb{F}_q)$ restricts to an isomorphism $R\rightarrow Q$. To see that $R$ is normal in $\tilde{J}$, note that $\tilde{Q} = R\mathbf{Z}(\GL_n(\mathbb{F}_q))$. Since $\mathbf{Z}(\GL_n(\mathbb{F}_q))$ normalizes $R$, $R$ is the normal Sylow $p$-subgroup of $\tilde{Q}$. In particular, $R$ is characteristic in $\tilde{Q}$, which is a normal subgroup of $\tilde{J}$ since $\overline{Q}$ is normal in $\overline{J}$. Therefore, $R$ is normal in $\tilde{J}$.

(b) The subgroup of upper triangular matrices with $1$'s on the diagonal is a Sylow $p$-subgroup of $\GL_n(\mathbb{F}_q)$, and it is obvious that every nontrivial subgroup of this Sylow $p$-subgroup has this property. Since $R$ is a nontrivial $p$-subgroup of $\GL_n(\mathbb{F}_q)$, it is conjugate to a nontrivial subgroup of this Sylow $p$-subgroup, so it also has this property.

(c) If $u\in (\mathbb{F}_q^n)^R$, then for any $r\in R$, we have $rh_\infty u= h_\infty(h_\infty^{-1}rh_\infty)u = h_\infty u$ since $h_\infty^{-1}rh_\infty\in R^{h_\infty} = R$. Therefore $h_\infty (\mathbb{F}_q^n)^R = (\mathbb{F}_q^n)^R$.
\end{proof}

This completes the proof of \autoref{gamma_0_gamma_infty}(b), so $h_\infty$ is as in \autoref{m2sp_elts}(c). To see which part of the spectrum of $h_\infty$ comes from the tame part, we will use \autoref{wild_basic_properties} and compare it to certain $\mathbb{C}R$-submodules of irreducible Weil modules $\Weil_j$ which we introduce in the next proposition. To see how these submodules interact with $h_\infty$, it is enough to work with a larger elementary abelian $p$-group containing $R$. For our $h_\infty$, choose $v, w\in \mathbb{F}_q^n\setminus\{0\}$ and $b,c, m\in \mathbb{Z}$ as in \autoref{m2sp_elts}(c), so that in the notation of section 2, $$h_\infty=\alpha_m^b \oplus \alpha_{n-m}^c\in \GL(V(h_\infty;v))\oplus \GL(V(h_\infty;w))\subset \GL_n(\mathbb{F}_q).$$ By exchanging $v$ and $w$ if necessary, we can assume that $(\mathbb{F}_q^n)^R=V(h_\infty;v)$. Define
\begin{align*}
E :&= \{r\in \GL_n(\mathbb{F}_q)\mid r\text{ acts trivially on both }V(h_\infty;v)\text{ and }\mathbb{F}_q^n/V(h_\infty;v)\}\\&= \left\lbrace \begin{pmatrix} I & X \\ 0 & I\end{pmatrix}\mid X\in M_{m\times(n-m)}(\mathbb{F}_q)\right\rbrace.
\end{align*}
Note that $E$ is an elementary abelian $p$-group containing $R$, and that $h_\infty$ normalizes $E$.

\begin{prop}\label{permute}
There is a set of $\frac{(q^m-1)(q^{n-m}-1)}{q-1}$ one-dimensional $\mathbb{C}E$-submodules of $\Weil_j$ which is cyclically permuted by $h_\infty$. If $j=0$, they are all contained in $\Weil_0'$. Also, if we choose one nonzero vector from each of them, then they are linearly independent.
\end{prop}

\begin{proof}
Let $v, w, b, c, m$ be as above. We may view $V(h_\infty;v)$ as an additive elementary abelian $p$-group of order $q^m$. For each irreducible $\mathbb{C}$-character $\varphi\in \Irr (V(h_\infty;v))$ and each $y\in V(h_\infty;w)\setminus\{0\}$, define $$y_{\varphi, j} := \sum_{x\in V(h_\infty;v)}\varphi(x)(x+y)^{(j)}\in \Weil_j. $$ Note that the map $E\rightarrow V(h_\infty;v)$ given by $r \mapsto ry-y$ is a surjective group homomorphism: if $r, r'\in E$, then they fix $V(h_\infty;v)=(\mathbb{F}_q^n)^R$ pointwise, and $r'y-y\in V(h_\infty,v)$, so $$rr'y-y = rr'y - ry + ry - y = r(r'y-y) + ry - y = (r'y-y) +(ry-y).$$ Therefore, $r\mapsto \varphi(ry-y)^{-1}$ is a one-dimensional $\mathbb{C}$-representation of $E$, and non-isomorphic pairs of $\varphi$ give non-isomorphic pairs of representations of $E$. Also, for each $r\in E$, we have 
\begin{align*}
ry_{\varphi,j} &= r\sum_{x\in V(h_\infty;v)}\varphi(x)(x+y)^{(j)} = \sum_{x\in V(h_\infty;v)}\varphi(x)(r(x+y))^{(j)} = \sum_{x\in V(h_\infty;v)}\varphi(x)(x+(ry-y)+y)^{(j)} \\&= \sum_{x\in V(h_\infty;v)}\varphi(x-(ry-y))(x+y)^{(j)} = \varphi(ry-y)^{-1}\sum_{x\in V(h_\infty;v)}\varphi(x)(x+y)^{(j)} = \varphi(ry-y)^{-1}y_{\varphi,j}.
\end{align*}
Hence, $\mathbb{C}y_{\varphi,j}$ is a $\mathbb{C}E$-module on which $E$ acts by the representation described above. Also, for each positive integer $d$, we have
\begin{align*}
h_\infty^d y_{\varphi,j} &= h_\infty^d \sum_{x\in V(h_\infty;v)}\varphi(x)(x+y)^{(j)} = \sum_{x\in V(h_\infty;v)}\varphi(x)(h_\infty^d(x+y))^{(j)} \\&= \sum_{x\in V(h_\infty;v)}\varphi(x)(\alpha_m^{bd}x+\alpha_{n-m}^{cd}y)^{(j)} = \sum_{x\in V(h_\infty;v)}\varphi(\alpha_m^{-bd}x)(x+\alpha_{n-m}^{cd}y)^{(j)}= (\alpha_{n-m}^{cd}y)_{\varphi\circ \alpha_m^{-bd}, j}
\end{align*}
which makes sense because $\varphi\circ\alpha_m^{-bd} : x \mapsto \varphi(\alpha_m^{-bd}x)$ is an irreducible character of $V(h_\infty;v)$. Of course, $\varphi\circ\alpha_m^{-bd}$ is nontrivial if and only if $\varphi$ is nontrivial. Therefore, $h_\infty$ permutes the set of $\frac{(q^m-1)(q^{n-m}-1)}{q-1}$ pairwise non-isomorphic one-dimensional $\mathbb{C}E$-modules $$\{\mathbb{C}y_{\varphi,j} \mid y\in V(h_\infty;w)\setminus\{0\},\ \varphi\in \Irr(V(h_\infty;v))\setminus \{\mathds{1}\}\}.$$ Note that when $j=0$, they are all nontrivial, so we can view them as submodules of $\Weil_0'$ (recall that $\Weil_0$ is a direct sum of a trivial module and the irreducible module $\Weil_0'$.) Also, since $\Irr(V(h_\infty;v))$ is linearly independent, so are the vectors $v_{\varphi,j}$.

I claim that this action of $h_\infty$ has only one orbit, so that it cyclically permutes these modules. Fix $y\in V(h_\infty;w)\setminus\{0\}$ and $\varphi\in \Irr(V(h_\infty;v))\setminus\{\mathds{1}\}$. Let $d$ be the size of the orbit containing $\mathbb{C}y_{\varphi,j}$. It is the smallest positive integer such that $h_\infty^d y_{\varphi,j} = (\alpha_{n-m}^{cd} y)_{\varphi\circ \alpha_m^{-bd},j} \in \mathbb{C}y_{\varphi,j}$. If we write this as a linear combination of the usual basis vectors of the form $u^{(j)}$, $u\in \mathbb{F}_q^n\setminus\{0\}$, then the coefficient of $y^{(j)}$ is nonzero if and only if $\alpha_{n-m}^{cd}y\in \mathbb{F}_qy$. Since $y_{\varphi,j}$ has nonzero coefficient for $y^{(j)}$ and $0\neq h_\infty^dy_{\varphi,j}\in \mathbb{C}y_{\varphi,j}$, the coefficient is indeed nonzero, and $\alpha_{n-m}^{cd}y\in \mathbb{F}_qy$. Recall that $\alpha = \alpha_{n-m}^{\frac{q^{n-m}-1}{q-1}}$ is the lowest power of $\alpha_{n-m}$ contained in $\mathbb{F}_q$. Since $c$ is relatively prime to $\frac{q^{n-m}-1}{q-1}$, we must have $d=d'\frac{q^{n-m}-1}{q-1}$ for some positive integer $d'$. Hence 
\begin{align*}
h_{\infty}^dy_{\varphi,j} &= (\alpha_{n-m}^{cd}y)_{\varphi\circ \alpha_m^{-bd},j} = (\alpha^{cd'}y)_{\varphi\circ \alpha_m^{-bd'\frac{q^{n-m}-1}{q-1}},j} = \sum_{x\in V(h_\infty;v)} \varphi(\alpha_m^{-bd'\frac{q^{n-m}-1}{q-1}}x)(x+\alpha^{cd'}y)^{(j)}\\&= \lambda^{cd'j}\sum_{x\in V(h_\infty;v)} \varphi(\alpha_m^{-bd'\frac{q^{n-m}-1}{q-1}}x)(\alpha_m^{-cd'\frac{q^{m}-1}{q-1} }x+y)^{(j)}= \lambda^{cd'j}y_{\varphi\circ \alpha_m^{-bd'\frac{q^{n-m}-1}{q-1} +cd'\frac{q^m-1}{q-1}},j} \in \mathbb{C}y_{\varphi,j}.
\end{align*}
Therefore, $d'$ must satisfy $\varphi\circ \alpha_m^{-bd'\frac{q^{n-m}-1}{q-1} +cd'\frac{q^m-1}{q-1}} = \varphi$, or equivalently, $$\alpha_m^{-bd'\frac{q^{n-m}-1}{q-1} +cd'\frac{q^m-1}{q-1}}x -x\in \ker\varphi\text{ for all }x\in V(h_\infty;v).$$ Since $\ker \varphi$ is a proper subgroup of $V(h_\infty;v)$, the linear transformation $\alpha_m^{-bd'\frac{q^{n-m}-1}{q-1} +cd'\frac{q^m-1}{q-1}}-I: V(h_\infty;v)\rightarrow \ker \varphi \subsetneq V(h_\infty;v)$ is not invertible. Therefore, $1$ is an eigenvalue of $ \alpha_m^{-bd'\frac{q^{n-m}-1}{q-1} +cd'\frac{q^m-1}{q-1}}$ as an element of $\GL_m(\mathbb{F}_q)$. Since $\alpha_m$ is a generator of $\mathbb{F}_{q^m}^\times$, the eigenvalues of $\alpha_m\in \GL_m(\mathbb{F}_q)$ are some primitive $q^m-1$th roots of unity. Therefore, $q^m-1$ must divide $(\frac{-b(q^{n-m}-1)+c(q^m-1)}{q-1})d'$. Recall that $\frac{b(q^{n-m}-1)-c(q^m-1)}{q-1}$ is relatively prime to $q-1$. It follows that $d' = d''(q-1)$ for some positive integer $d''$, and $\frac{q^m-1}{q-1}$ divides $(\frac{-b(q^{n-m}-1)+c(q^m-1)}{q-1})\frac{d'}{q-1} = -\frac{bd''(q^{n-m}-1)}{q-1} + cd''\frac{q^m-1}{q-1}$. We also know that $\frac{q^m-1}{q-1}$ is relatively prime to both $b$ and $\frac{q^{n-m}-1}{q-1}$, so $\frac{q^m-1}{q-1}$ divides $d''$. Therefore, $d= \frac{q^{n-m}-1}{q-1}d'= (q^{n-m}-1)d''$ is divisible by $\frac{(q^m-1)(q^{n-m}-1)}{q-1}$. On the other hand, $h_\infty^{\frac{(q^m-1)(q^{n-m}-1)}{q-1}}=1$, so $d$ is exactly this number. Therefore, $h_\infty$ cyclically permutes the $d$ one-dimensional $\mathbb{C}E$-modules $\mathbb{C}y_{\varphi,j}$ with $y\in V(h_\infty;w)\setminus \{0\}$ and $\varphi\in\Irr(V(h_\infty;v))\setminus\{\mathds{1}\}$. 
\end{proof}

\begin{thm}\label{wild_dim}
$\dim \Wild = \frac{(q^m-1)(q^{n-m}-1)}{q-1}$, and the spectrum of the action of $\gamma_\infty$ on $\Wild$ is the $(\dim \Wild)$th roots of some number. The eigenvalues of the action of $\gamma_\infty$ on $\Tame$ have multiplicity $1$, and they are precisely the eigenvalues of the action of $\gamma_\infty$ on $\mathcal{H}$ with multiplicity $2$.
\end{thm}

\begin{proof}
Recall that $\mathcal{H}$ and an irreducible Weil module $W$ of $\GL_n(\mathbb{F}_q)$ give the same projective representation of $\PGL_n(\mathbb{F}_q)$. Since $R$ is abelian, the irreducible representations of $R$ are one-dimensional. By \autoref{wild_basic_properties}, $\dim \Wild$ cannot be divisible by $p$, so $\Wild$ is a direct sum of $\dim \Wild$ one-dimensional $P(\infty)$-representations cyclically permuted by $\gamma_\infty$. In particular, the spectrum of the action of $\gamma_\infty$ on $\Wild$ is the set of $(\dim \Wild)$th roots of some number.

By \autoref{permute}, we know that $h_\infty$ cyclically permutes $\frac{(q^m-1)(q^{n-m}-1)}{q-1}$ one-dimensional $\mathbb{C}R$-submodules of $W$, and they generate a $\frac{(q^m-1)(q^{n-m}-1)}{q-1}$-dimensional $\mathbb{C}\tilde{J}$-submodule of $W$. If we choose one of these one-dimensional $\mathbb{C}R$-submodule and choose the corresponding one-dimensional $P(\infty)$-submodule of $\mathcal{H}$, then it must be contained in either $\Tame$ or $\Wild$. Hence, either $\Tame$ or $\Wild$ has dimension at least $\frac{(q^m-1)(q^{n-m}-1)}{q-1}$. On the other hand, the spectrum of the action of $\gamma_\infty$ on $\mathcal{H}$ (which is just a root of unity times the spectrum of the action of $h_\infty$ on $W$) has $\frac{q^m-1}{q-1}+\frac{q^{n-m}-1}{q-1}-\delta_{j,0}$ eigenvalues of multiplicity $2$, and each of them must appear exactly once on both $\Tame$ and $\Wild$. Therefore, one of $\Tame$ and $\Wild$ has dimension $\frac{(q^m-1)(q^{n-m}-1)}{q-1}$ and the other has dimension $\frac{q^m-1}{q-1}+\frac{q^{n-m}-1}{q-1}-\delta_{j,0}$. Moreover, these eigenvalues are some root of unity times the $\frac{q^m-1}{q-1}$th roots of $\lambda^{bj}$ and the $\frac{q^{n-m}-1}{q-1}$th roots of $\lambda^{cj}$. By the observation in the previous paragraph, they all have the same $(\dim \Wild)$th power. In particular, $\dim \Wild$ is divisible by $\lcm(\frac{q^m-1}{q-1},\frac{q^{n-m}-1}{q-1}) = \frac{(q^m-1)(q^{n-m}-1)}{(q-1)^2}$, so $\dim \Wild > \frac{q^m-1}{q-1}+\frac{q^{n-m}-1}{q-1}$. Therefore $\dim \Wild = \frac{(q^m-1)(q^{n-m}-1)}{q-1}$.
\end{proof}

\section{Candidate Hypergeometric Sheaves}

In the previous two sections, we found a complete list of possible spectra of $\gamma_0$ on $\mathcal{H}$ and that of $\gamma_\infty$ on $\Tame$ and $\Wild$. In this section, we will show that only a small number of pairs of these spectra for $\gamma_0$ and $\gamma_\infty$ can occur together as the spectra of $\gamma_0$ and $\gamma_\infty$ on $\mathcal{H}$. 

Let $\mathcal{H}=\mathcal{H}yp_\psi(\chi_1,\dots,\chi_D; \rho_1,\dots,\rho_{D-W})$ be a hypergeometric sheaf with the wild part $\Wild$ at $\infty$ of dimension $W$ and the geometric monodromy group $G$, and suppose that \eqref{hyp_condition} holds. The geometric determinant of $\mathcal{H}$ is given in \cite[Lemma 8.11.6]{K-ESDE} as
\begin{equation}\label{geom_det}
\det(\mathcal{H}) = \begin{cases} \mathcal{L}_{\prod_{i=1}^{D}\chi_i} & \text{if }W>1,\\ \mathcal{L}_{\prod_{i=1}^{D}\chi_i}\otimes \mathcal{L}_{\psi} = \mathcal{K}l_{\psi}(\prod_{i=1}^{D}\chi_i) & \text{if }W=1. \end{cases}
\end{equation}
Here, $\mathcal{L}_{\prod_{i=1}^{D}\chi_i}$ is the Kummer sheaf defined by the multiplicative character $\prod_{i=1}^{D}\chi_i$, and $\mathcal{L}_{\psi}$ is the Artin-Schreier sheaf defined by the additive character $\psi$. In the next proposition, we use this determinant to see how the upstairs and downstairs characters are related.

We first treat the cases where $D=\frac{q^n-1}{q-1}$. Note that the irreducible Weil representations of these dimensions are imprimitive, so the corresponding sheaves will also be imprimitive. It turns out that they must be Belyi induced. 

For a positive integer $N$ and a multiplicative character $\chi$, we define $$\Char(N, \chi) := \{\text{multiplicative characters }\varphi\text{ of }\overline{\mathbb{F}_p}\text{ with }\varphi^N=\chi\}.$$ If $\chi=\mathds{1}$, we will also write $\Char(N)$ instead of $\Char(N,\mathds{1})$. 

\begin{thm}\label{large_rank}
Let $\mathcal{H}$ and $G$ be as above, and suppose that $D=\frac{q^n-1}{q-1}$. Then $\mathcal{H}$ must be of the form
\begin{align*}
&\mathcal{H}yp_\psi(\varphi\Char(\frac{q^n-1}{q-1},\chi^{(b+c)j}); \varphi\Char(\frac{q^m-1}{q-1}, \chi^{bj}) \cup \varphi\Char(\frac{q^{n-m}-1}{q-1}, \chi^{cj}))\\\cong &  \mathcal{H}yp_\psi(\Char(\frac{q^n-1}{q-1},\chi^{(b+c)j}); \Char(\frac{q^m-1}{q-1}, \chi^{bj}) \cup \Char(\frac{q^{n-m}-1}{q-1}, \chi^{cj}))\otimes \mathcal{L}_\varphi
\end{align*}
for some nontrivial additive character $\psi$ of $\mathbb{F}_p$, some multiplicative character $\varphi$ of a finite extension of $\mathbb{F}_q$, an integer $1\leq j\leq q-2$, integers $m,b,c$ as in \autoref{m2sp_elts}(c), and a multiplicative character $\chi$ of $\mathbb{F}_q^\times$ of order $q-1$. Moreover, we can assume that $b\frac{q^{n-m}-1}{q-1} - c\frac{q^m-1}{q-1}=1$ and $\gcd(\frac{q^n-1}{q-1}, \frac{q-1}{\gcd(q-1,c)})=1$.
\end{thm}

\begin{proof}
Let $K$ be a finite extension of $\mathbb{F}_q$ over which all upstairs characters $\chi_i$ and downstairs characters $\rho_j$ are defined, that is, $\#K-1$ is divisible by the orders of all $\chi_i$ and $\rho_j$. Let $\vartheta$ be a multiplicative character of order $\#K-1$ of $K$, so that every upstairs and downstairs character is a power of $\vartheta$. Let $\Theta:\pi_1^{et}(\mathbb{G}_m/K)\rightarrow \overline{\mathbb{Q}_\ell}$ be the monodromy representation of the Kummer sheaf $\mathcal{L}_\vartheta$. Then both $\Theta(\gamma_0)$ and $\Theta(\gamma_\infty)$ are primitive $(\#K-1)$th roots of unity, so we can fix an integer $r$ relatively prime to $\#K-1$ such that $\Theta(\gamma_0)= \Theta(\gamma_\infty)^r$. The same relation holds (with the same $r$) for the values at $\gamma_0$ and $\gamma_\infty$ of the monodromy representations of the Kummer sheaves for all powers of $\vartheta$.

To apply the results of the previous section, we first prove that $W=\dim \Wild>1$. So suppose that $W=1$. By \eqref{geom_det}, the geometric determinant of $\mathcal{H}$ is $\mathcal{L}_{\prod_{i=1}^{D}\chi_i}\otimes \mathcal{L}_\psi$. Note that $\psi$ has order $p$ while both $\gamma_0$ and $\gamma_\infty$ have pro-order prime to $p$. Also, since each of $\chi_i$ is a power of $\vartheta$, so is $\prod_{i=1}^{D}\chi_i$. Therefore, for the images $g_0,g_\infty\in G\leq \GL_D(\overline{\mathbb{Q}_\ell})$ of $\gamma_0,\gamma_\infty$, we have $$\det g_0 = (\det g_\infty)^r.$$

Since $\mathcal{H}$ satisfies \eqref{hyp_condition} and $D=\frac{q^n-1}{q-1}$ , the restriction of $\mathcal{H}$ to $E(G)$ comes from $\Weil_j$ for some $j\in \{1,\dots,q-2\}$.  By \autoref{gamma_0_gamma_infty} and \autoref{m2sp_elts}, the eigenvalues of $g_0$ and $g_\infty$ are given by $$\{\mu\zeta\mid \zeta^{\frac{q^n-1}{q-1}} = \lambda^{aj}\}\text{ and }\{\nu\zeta \mid \zeta^{\frac{q^n-1}{q-1}} = \lambda^{bj}\},\text{ respectively,}$$ where $a,b$ are integers relatively prime to $\frac{q^n-1}{q-1}$, and $\mu,\nu\in \overline{\mathbb{Q}_\ell}^\times$ are some roots of unity. From the above equality of determinants, we get $$\mu^{\frac{q^n-1}{q-1}}\lambda^{aj} = (\nu^{\frac{q^n-1}{q-1}}\lambda^{bj})^r,\text{ so that }\{(\nu\zeta)^r \mid \zeta^{\frac{q^n-1}{q-1}}=\{\mu\zeta \mid \zeta^{\frac{q^n-1}{q-1}} = \lambda^{aj}\}.$$
Note that the set on the left-hand side contains the $r$th powers of values at $\gamma_\infty$, which are equal to the values at $\gamma_0$, of the monodromy representations of Kummer sheaves obtained from the downstairs characters. The set on the right-hand side is precisely the set obtained similarly from the upstairs characters. Therefore, the set of downstairs characters is contained in the set of upstairs characters, which is impossible by \autoref{basic_facts}(a). Therefore $W>1$.

Now we can apply \autoref{basic_facts}, \autoref{gamma_0_gamma_infty} and \autoref{wild_dim} to find the possible spectra of $g_0$ and $g_\infty$, as well as the possible sets of upstairs and downstairs characters: the spectra must be $$\{\mu\zeta\mid \zeta^{\frac{q^n-1}{q-1}}=\lambda^{aj}\}$$and$$\{\nu\zeta\mid \zeta^{\frac{q^m-1}{q-1}}=\lambda^{bj}\}\sqcup\{\nu\zeta\mid \zeta^{\frac{q^{n-m}-1}{q-1}}=\lambda^{cj}\}\sqcup\{\nu\zeta\mid \zeta^{\frac{(q^m-1)(q^{n-m}-1)}{q-1}}=1\},$$ and the sets of upstairs and downstairs characters are $$\eta\Char(\frac{q^n-1}{q-1},\chi^{aj})\text{ and }\varphi\Char(\frac{q^m-1}{q-1},\chi^{bj})\cup\varphi\Char(\frac{q^{n-m}-1}{q-1},\chi^{cj}),$$ respectively, for some integers $a,b,c,m$ satisfying the conditions in \autoref{m2sp_elts} and some roots of unity $\mu, \nu\in \overline{\mathbb{Q}_\ell}^\times$ and the corresponding multiplicative characters $\eta,\varphi$ of some finite extension of $\mathbb{F}_q$. By computing $\det g_0 = (\det g_\infty)^r$ as above, we can see that $$\eta^{\frac{q^n-1}{q-1}}\chi^{aj} = \varphi^{\frac{q^n-1}{q-1}}\chi^{(b+c)j}$$ so that the upstairs characters can be written as $$\eta\Char(\frac{q^n-1}{q-1},\chi^{aj}) = \varphi\Char(\frac{q^n-1}{q-1},\chi^{(b+c)j}).$$ Therefore $\mathcal{H}$ must be of the form
\begin{align*}
\mathcal{H}yp_\psi(\varphi\Char(\frac{q^n-1}{q-1},\chi^{(b+c)j}); \varphi\Char(\frac{q^m-1}{q-1}, \chi^{bj}) \cup \varphi\Char(\frac{q^{n-m}-1}{q-1}, \chi^{cj})).
\end{align*}
By \cite[8.2.14]{K-ESDE} this is geometrically isomorphic to
$$  \mathcal{H}yp_\psi(\Char(\frac{q^n-1}{q-1},\chi^{(b+c)j}); \Char(\frac{q^m-1}{q-1}, \chi^{bj}) \cup \Char(\frac{q^{n-m}-1}{q-1}, \chi^{cj}))\otimes \mathcal{L}_\varphi.$$

Let $d$ be a positive integer such that $q-1$ is not divisible by the $d$th power of any prime. Since $m$ is relatively prime to $n$, we can choose integers $e, x, y$ such that $$e(n-m)\equiv c\text{ mod }n^d,\text{ and } x\frac{q^{n-m}-1}{q-1} - y\frac{q^m-1}{q-1}=1.$$ Since $b\frac{q^{n-m}-1}{q-1}-c\frac{q^m-1}{q-1}$ is relatively prime to $q-1$, we can also choose integers $z,w$ such that $$z(b\frac{q^{n-m}-1}{q-1}-c\frac{q^m-1}{q-1}) - w(q-1) = 1.$$ Let $b':=z(b-e\frac{q^m-1}{q-1})-(q-1)xw$, $c':=z(c-e\frac{q^{n-m}-1}{q-1})-(q-1)yw$, $j':=(b(n-m)-cm)j$, and $\varphi':=\varphi\chi^{ej}$. One can easily check that in the above expression of $\mathcal{H}$, replacing the parameters $(b, c, j, \varphi)$ with $(b',c',j',\varphi')$ does not change the set of upstairs and downstairs characters, and these new parameters satisfy $b'\frac{q^{n-m}-1}{q-1}-c'\frac{q^m-1}{q-1}=1$ and $\gcd(\frac{q^n-1}{q-1}, \frac{q-1}{\gcd(q-1,c')})=\gcd(n, \frac{q-1}{\gcd(q-1,c-e(n-m))}) =1$.
\end{proof}

The sheaves of rank $\frac{q^n-q}{q-1}$ require more work. In particular, we will need to use the so-called ``$V$-test'', which is a criterion determining whether the geometric monodromy group of an irreducible hypergeometric sheaf is finite or not, based on an inequality involving the upstairs and downstairs characters and Kubert's $V$ function. For details and the proof of this test, see \cite[Section 13]{K-G2} and \cite[Section 8.16]{K-ESDE}. We shall also use the basic properties \cite[Section 13, p. 206]{K-G2} of the function $V$, without explicit mention, to simplify expressions involving $V$.

\begin{thm}\label{small_rank}
Let $\mathcal{H}$ and $G$ be as in the beginning of this section, so that \eqref{hyp_condition} holds. Suppose that $D=\frac{q^n-q}{q-1}$. Then $\mathcal{H}$ must be of the form 
\begin{align*}
&\mathcal{H}yp_\psi(\varphi(\Char(\frac{q^n-1}{q-1})\setminus\{\mathds{1}\}); \varphi\Char(\frac{q^m-1}{q-1}) \cup \varphi(\Char(\frac{q^{n-m}-1}{q-1})\setminus\{\mathds{1}\}))\\\cong& \mathcal{H}yp_\psi(\Char(\frac{q^n-1}{q-1})\setminus\{\mathds{1}\}; \Char(\frac{q^m-1}{q-1}) \cup (\Char(\frac{q^{n-m}-1}{q-1})\setminus\{\mathds{1}\}))\otimes \mathcal{L}_\varphi
\end{align*}
for some nontrivial additive character $\psi$ of $\mathbb{F}_p$, some multiplicative character $\varphi$ of a finite extension of $\mathbb{F}_q$, and an integer $m$ satisfying the conditions in \autoref{m2sp_elts}(c).
\end{thm}

\begin{proof}
(1) We first prove that $W>1$. \\

Choose $K, \vartheta, \Theta$ and $r$ as in the proof of \autoref{large_rank}. We can assume that $\#K-1$ is also divisible by $\frac{q^{n-1}-1}{q-1}$. As before, the elements $g_0$ and $g_\infty$ must be as in \autoref{m2sp_elts}(a). However, \autoref{m2sp_elts} describes the spectrum of $\Weil_0=\overline{\mathbb{Q}_\ell}\oplus\Weil_0'$, while $\mathcal{H}$ gives $\Weil_0'$. Therefore, the spectra of $g_0$ and $g_\infty$ can be written as $$\{\mu\zeta\mid \zeta^{\frac{q^n-1}{q-1}}=1,\zeta\neq 1\}\text{ and }\{\nu\zeta\mid \zeta^{\frac{q^n-1}{q-1}}=1, \zeta\neq 1\}$$ for some roots of unity $\mu,\nu\in \overline{\mathbb{Q}_\ell}^\times$ of order not divisible by $p$. Now from the equality $\det g_0 = (\det g_\infty)^r$, we get $$\mu^{\frac{q^n-q}{q-1}} = \nu^{\frac{(q^n-q)r}{q-1}},\text{ or equivalently }(\mu\nu^{-r})^{\frac{q^{n-1}-1}{q-1}}=1.$$ Therefore, $\mathcal{H}$ must be geometrically isomorphic to 
\begin{equation}\label{W=1_sheaf}
\mathcal{H}yp_\psi(\tau(\Char(\frac{q^n-1}{q-1})\setminus\{\mathds{1}\}); \Char(\frac{q^n-1}{q-1})\setminus \{\mathds{1},\rho\})\otimes \mathcal{L}_\varphi
\end{equation} for some nontrivial multiplicative character $\tau$ of order dividing $\frac{q^{n-1}-1}{q-1}$, some $\mathds{1}\neq \rho\in \Char(\frac{q^n-1}{q-1})$, and some multiplicative character $\varphi$. Here, $\tau$ is nontrivial since $\mathcal{H}$ is irreducible so that the upstairs and downstairs characters must be disjoint. 

I claim that the hypergeometric sheaf \eqref{W=1_sheaf} does not satisfy \eqref{hyp_condition}. More specifically, I will show that the geometric monodromy group is not finite. For the sheaf \eqref{W=1_sheaf}, the $V$-test, after a simplification, says that the geometric monodromy group is finite if and only if for every integer $N$ relatively prime to $\#K-1$ and for every $x\in \mathbb{Q}/\mathbb{Z}$ whose denominator is not divisible by $p$, the following inequality holds:
\begin{equation}\label{eq:Vtest_D-m=1}
\begin{split}
V(Nt + Ax)  + V(-Ax) +\frac{3}{2}  \geq & \frac{(A-2)V(Nt)+V(Nt-\frac{Ns}{A})}{A-1}+ V(Nt+x) + V(-x) + V(-\frac{Ns}{A}-x)
\end{split}
\end{equation}
where $A=\frac{q^n-1}{q-1}$, and $t\in \frac{q-1}{q^{n-1}-1}\mathbb{Z}\setminus \mathbb{Z}$ is the number such that $\vartheta^{t(\#K-1)}  = \tau$, and $s\in \{1,\dots, A-1\}$ is the number such that $\vartheta^{s\frac{(\#K-1)}{A} }=\rho$.

Suppose that \eqref{eq:Vtest_D-m=1} holds for all pairs $(N,x)$. Recall from \cite[Section 13]{K-G2} that for any $x\in \mathbb{Q}/\mathbb{Z}\setminus \mathbb{Z}$, we have $V(x)+V(-x)=1$, and for any $x\in \mathbb{Z}$ we have $V(x)=0$. Note that for every $N$ relatively prime to $\#K-1$ (which is divisible by both $A$ and $\frac{q^{n-1}-1}{q-1}$ by the choice of $K$), if we let $x=-Nt+\frac{u}{A}$ for an integer $u$ not divisible by $A$, then we have $Nt+Ax=-(A-1)Nt+u\in \mathbb{Z}$ but $Ax, Nt, Nt-N\frac{s}{A}, Nt+x, x, N\frac{s}{A}-x$ are not integers, so that the sum of \eqref{eq:Vtest_D-m=1} for the pairs $(N,x)$ and $(-N,-x)$ becomes the inequality $1 \geq 4-3$. Therefore the equality holds in \eqref{eq:Vtest_D-m=1} for each of these pairs:
\begin{align*}
&\frac{V(Nt) - V(Nt-\frac{Ns}{A})}{A-1}  = V(\frac{u}{A}) + V(Nt-\frac{u}{A}) + V(Nt-\frac{Ns+u}{A}) - \frac{3}{2}.
\end{align*}
The explicit formula \cite[Theorem 13.4]{K-G2} of the function $V$ tells us that the right-hand side of the equality lies in $\frac{1}{(p-1)n(n-1)\log_p q}\mathbb{Z}$. On the other hand, the left-hand side lies in $(-\frac{1}{A-1}, \frac{1}{A-1})$. Since $n\geq 3$, we have $A-1=\frac{q^n-q}{q-1}\geq (p-1)n(n-1)\log_p q$, so that $\frac{1}{(p-1)n(n-1)\log_p q}\mathbb{Z}\cap (-\frac{1}{A-1}, \frac{1}{A-1})=\{0\}$. Therefore, the equality forces that $$V(Nt) = V(Nt-\frac{Ns}{A})\text{, and }V(\frac{u}{A}) + V(Nt-\frac{u}{A}) + V(Nt-\frac{Ns+u}{A}) = \frac{3}{2}.$$ These are true for all integers $N$ relatively prime to $\#K-1$, so \eqref{eq:Vtest_D-m=1} implies
\begin{equation}\label{eq:Vtest_var}\tag{\ref*{eq:Vtest_D-m=1}${}^\prime$}
\begin{split}
&V(Nt + Ax)  + V(-Ax) + \frac{3}{2}   \geq  V(Nt) + V(Nt+x) + V(-x) + V(-\frac{Ns}{A}-x)
\end{split}
\end{equation}
for all pairs $(N,x)$. 

The equality in \eqref{eq:Vtest_var} holds not only for $x=-Nt+\frac{u}{A}$ for $u\in \mathbb{Z}\setminus A\mathbb{Z}$, but also for $x=\frac{u}{A}$ for $u\in \mathbb{Z}$ with $u\not\equiv 0,-Ns$ mod $A$ by the same reason. For $x=-Nt+\frac{u}{A}$ and $-\frac{u}{A}$ with $u\not\equiv 0,Ns$ mod $A$, these become 
\begin{equation}
\begin{split}
\frac{3}{2}&= V(\frac{u}{A})+V(Nt-\frac{u}{A}) + V(Nt-\frac{u+Ns}{A})=V(Nt-\frac{u}{A})+V(\frac{u}{A})+V(\frac{u-Ns}{A})
\end{split}\label{eq:3/2}
\end{equation}
so that 
\begin{equation}
V(Nt-\frac{u+Ns}{A}) = V(\frac{u-Ns}{A}).\label{eq:eqty}
\end{equation} 

Take the sum of \eqref{eq:eqty} for $u\in \{1, \dots, A-1\}$ except for $Ns$ mod $A$, and use the fact $V(Nt)= V(Nt-\frac{Ns}{A})$ we saw above to get:
$$V(Nt-\frac{2Ns}{A}) = V(-\frac{Ns}{A}).$$ On the other hand, choose $u=Ns+p^f$ for either $f=0$ or $1$, so that $Ns+p^f\not\equiv 0$ mod $A$. Then from \eqref{eq:eqty} we get $$\frac{2}{n}=2V(\frac{p^f}{A}) =V(Nt-\frac{u+Ns}{A})+V(\frac{u-Ns}{A})\geq V(Nt-\frac{2Ns}{A})=V(-\frac{Ns}{A}).$$ Similarly $\frac{2}{n}\geq V(\frac{Ns}{A})$, so we must have $\frac{4}{n}\geq V(-\frac{Ns}{A})+V(\frac{Ns}{A})=1$. Therefore, either $n=4$ and $V(\frac{Ns}{A})=1/2$, or $n=3$ and $1/3\leq V(\frac{Ns}{A})\leq 2/3$.

Suppose that $n=3$. Let $N=1$, $t' = t(q+1)\in \{1,\dots, q\}$, and $x = \frac{(q+1-t')(q-1)(q^{n(n-2)}-1)}{(q^n-1)(q^{(n-1)^2}-1)} =\frac{(q- t')q + t'-1}{q^4-1} $. Then  
\begin{align*}
2V(t)&=V(t)+V(qt)=V(t)+V(t'-t)=1,\\V(x)&= V(\frac{(q-t')q + t'-1}{q^4-1})= V(\frac{q-t'}{q^4-1})+V(\frac{t'-1}{q^4-1})=\frac{1}{4} ,\\V(Ax) &=V(\frac{(q-t')q^3+(q^2-1)q+(t'-1)}{q^4-1})= \frac{3}{4},\\V(t+Ax) &=V(\frac{(q-t')q^2 +(t'-1)q}{q^4-1}) =\frac{1}{4},\\V(t+x)&= V(\frac{(t'-1)q^3 + (q-t')q^2+q^2-1}{q^4-1})=\frac{3}{4}.
\end{align*}
Therefore \eqref{eq:Vtest_var} for $(N,x)$ becomes $2 \geq 2 + V(-\frac{s}{A}-x)$, which forces $x+\frac{s}{A}\in \mathbb{Z}$. However, this would imply $V(\frac{s}{A})= V(-x)=\frac{3}{4}$, but we saw above that $\frac{1}{3}\leq V(\frac{Ns}{A})\leq \frac{2}{3}$. Therefore $n\neq 3$. 

Finally, suppose that $n=4$. If $u\in \mathbb{Z}$ with $u\not\equiv 0,-Ns, -2Ns, -3Ns$ mod $A$, then by the same argument used in \eqref{eq:3/2} and \eqref{eq:eqty} with $x=-Nt+\frac{u}{A}, -\frac{u+Ns}{A}, -Nt+\frac{u+Ns}{A}, -Nt+\frac{u+2Ns}{A},$ and $-\frac{u+3Ns}{A}$, we get $$V(\frac{u}{A}) =  V(Nt-\frac{u+2Ns}{A}) = V(\frac{u+3Ns}{A})\text{ and }V(Nt-\frac{u}{A})=V(\frac{u+Ns}{A}).$$ Suppose that $d:=\gcd(3s, A)>1$. Then we can set $u=1, 1+3Ns, 1+6Ns, \dots$ to get $$\frac{1}{4}= V(\frac{1}{A})=V(\frac{1+d}{A})=\cdots = V(\frac{1+(A-d)}{A})$$ so that $$\frac{A}{4d} = \sum_{i=0}^{\frac{A}{d}-1}V(\frac{1+id}{A}) = V(\frac{1}{d}) + \frac{A-d}{2d}.$$ Therefore, $2d=4dV(\frac{1}{d}) + A$, so we must have $2d>A$. Since $d$ divides $A$, we get $d=A$, so that $A$ divides $3s$. Therefore $A$ is divisible by $3$, and $s=A/3$ or $2A/3$. By \eqref{eq:3/2}, \eqref{eq:eqty} and the above observations, we get 
\begin{align*}
\frac{3}{2}&=V(\frac{1}{A})+V(Nt-\frac{1}{A})+V(Nt-\frac{1+Ns}{A})=V(\frac{1}{A}) + V(\frac{1}{A}+\frac{1}{3}) + V(\frac{1}{A} + \frac{2}{3})= V(\frac{3}{A})+1.
\end{align*} 
Therefore $\frac{1}{2}= V(\frac{3}{A})=V(\frac{3(q-1)}{q^4-1})$, so we must have $q=2$. Now one can manually check that for $q=2$ and $n=4$, the inequality \eqref{eq:Vtest_var} fails for all possible pairs of $t$ and $s$; for instance one can choose $(N,x) = (1, \frac{1}{15})$ for $s=5$ and all $t$, and $(N,x)=(1,\frac{2}{15})$ for $s=10$ and all $t$.

If $d=1$, then $A$ is not divisible by $3$, and $s$ is relatively prime to $A$. Hence either $q=3$ or $q\equiv 1$ mod $3$, so $A\equiv 1$ mod $3$. Then as above we get 
\begin{align*}
&V(\frac{3s}{A}) = V(\frac{6s}{A}) = \cdots = V(\frac{(A-1)s}{A})= 1-V(\frac{(3A-3)s}{A}),\\&V(\frac{(A+2)s}{A})=V(\frac{(A+5)s}{A}) = \cdots = V(\frac{(2A-2)s}{A})=1-V(-\frac{2s}{A})=\frac{1}{2}, \\&V(\frac{(2A+1)s}{A})=V(\frac{(2A+4)s}{A}) = \cdots = V(\frac{(3A-3)s}{A}).
\end{align*}
Moreover, since $V(\frac{1}{A})=\frac{1}{4}$ must be one of the values, the top and bottom rows are either $\frac{1}{4}$ or $\frac{3}{4}$. However $V(\frac{(2A+1)s}{A})= V(\frac{s}{A})=\frac{1}{2}$ as we discussed below \eqref{eq:eqty}, so this case cannot happen. This completes the proof that $W>1$.\\

(2) Now we consider the sheaves with $W>1$. \\

As in \autoref{large_rank}, $\mathcal{H}$ must be geometrically isomorphic to 
\begin{align*}
\mathcal{H}yp_\psi(\tau(\Char(\frac{q^n-1}{q-1})\setminus\{\mathds{1}\}); \Char(\frac{q^m-1}{q-1}) \cup (\Char(\frac{q^{n-m}-1}{q-1})\setminus\{\mathds{1}\}))\otimes \mathcal{L}_\varphi
\end{align*}
for some multiplicative character $\tau$ of order dividing $\frac{q^{n-1}-1}{q-1}$ and some multiplicative character $\varphi$. We want to show that $\mathcal{H}$ must be of this form with $\tau=\mathds{1}$. If $m=n-1$ or $1$, then we can replace $\tau$ with $\mathds{1}$ and $\varphi$ with $\varphi\tau$ without changing the sheaf itself, since the set of downstairs characters $\Char(\frac{q^{n-1}-1}{q-1})$ remains the same when multiplied by $\tau^{-1}$.

Suppose that $1<m<n-1$ and $\tau\neq \mathds{1}$. We may replace $m$ with $n-m$ without changing the sheaf, so assume that $n/2< m<n-1$. The $V$-test for this sheaf (with $\varphi=\mathds{1}$) is
\begin{equation}
\begin{split}
& V(Nt + Ax)  +  V(-Bx)  + V(-Cx) \geq  V(Nt+x) + V(-x)                            \label{eq:Vtest_D-m>1}
\end{split}
\end{equation}
for the pairs $(N,x)$, where $A=\frac{q^n-1}{q-1}$, $B=\frac{q^m-1}{q-1}$, $C=\frac{q^{n-m}-1}{q-1}$, and $t\in \frac{q-1}{q^{n-1}-1}\mathbb{Z}\setminus \mathbb{Z}$ is the number corresponding to $\tau$. Note that $A=B+Cq^m=Bq^{n-m}+C$. Suppose that this holds for all pairs $(N,x)$. 

Let $x= \frac{1}{B}$ and $N=q^f$ for $f=0,1,\dots,m-1$. Then \eqref{eq:Vtest_D-m>1} for $(N,x)$ becomes 
\begin{equation}\label{eq:Vtest_eq}
V(q^ft+\frac{C}{B}) \geq V(q^ft + \frac{1}{B})+ \frac{n-m-1}{m}.
\end{equation}  On the other hand, by the basic properties of $V$, $$V(q^ft+\frac{C}{B}) \leq V(q^ft + \frac{1}{B}) + V(\frac{C-1}{B}) = V(q^ft+\frac{1}{B})+\frac{n-m-1}{B}.$$ Therefore, the equality must hold. 

Let $t'\in \{1,\dots, \frac{q^{n-1}-1}{q-1}-1\}$ be the number such that $t' \equiv t\frac{q^{n-1}-1}{q-1}$ mod $\frac{q^{n-1}-1}{q-1}$. Let $a_i\in \{1,\dots, p-1\}$ be the base $p$ digits of $t'(q-1)$, that is, the unique numbers such that $\sum_{i=0}^{(n-1)\log_p q-1}a_ip^i = t'(q-1)$. Now from \eqref{eq:Vtest_eq} (with $f=0$) and the formula for the function $V$ (cf. \cite[Sections 2 and 4]{KRL} and \cite[Discussion above Theorem 2.9]{KRLT1}), one can see that if $b_i\in \{0,\dots, p-1\}, 0\leq i\leq (n-1)m\log_p q-1$ are the base $p$ digits of $$\frac{q^{(n-1)m}-1}{q^{n-1}-1}(\sum_{i=0}^{(n-1)\log_p q-1}a_ip^i)+ (q-1)\frac{q^{(n-1)m}-1}{q^{m}-1}=\sum_{i=0}^{(n-1)m\log_p q-1}b_ip^i ,$$ then for each $j=0,\dots, n-2$, we must have $$b_{(jm+1)\log_p q}=\cdots= b_{(jm+n-m)\log_p q-1} = 0.$$ It follows that either
\begin{itemize}
\item $a_{0}=\cdots=a_{(n-m)\log_p q-1}=0$, or
\item $a_{\log_p q}=\cdots = a_{(n-m)\log_p q-1}=p-1$ and at least one of $a_0,\dots, a_{\log_p q-1}$ is nonzero.
\end{itemize}
Note that if we choose a larger power of $p$ as $N$, then it ``circularly shifts'' the base $p$ digits of $t'(q-1)$. Hence, the above argument for different values of $f$ shows that either $a_0=\cdots=a_{(n-1)\log_p q-1}=0$ or $a_0=\cdots=a_{(n-1)\log_p q-1}=p-1$. But then $t'(q-1)=0$ or $\frac{q^{n-1}-1}{q-1}$, which contradicts our choice of $t$. Therefore \eqref{eq:Vtest_eq} cannot hold for $(N,x)=(q^f, \frac{1}{B})$ for some $f$, and hence \eqref{eq:Vtest_D-m>1} fails for this pair, contradicting our assumption. This proves that if $1<m<n-1$ and $\mathcal{H}$ has finite geometric monodromy group, then $\tau=\mathds{1}$.
\end{proof}

\begin{rmk}
Note that in \eqref{eq:Vtest_D-m>1}, if $\tau=\mathds{1}$, then $t=0$ and the inequality becomes $$V(Ax)+V(-Bx)+V(-Cx)\geq 1$$ for $x\notin \mathbb{Z}$. Since $V(Ax)+V(-Bx)+V(-Cx)\geq V(Ax)+V(-Bx-q^mCx)= V(Ax)+V(-Ax)=1$, this inequality holds. Therefore, the sheaves in \autoref{small_rank} do have finite geometric monodromy groups.
\end{rmk}

By \autoref{large_rank} and \autoref{small_rank}, we are left with a small family of hypergeometric sheaves. Is it possible to further reduce this family? First, taking tensor product by a Kummer sheaf $\mathcal{L}_\varphi$ wouldn't make much difference on the geometric monodromy group, so we can't remove this. Indeed, if $G$ is the geometric monodromy group of the hypergeometric sheaf without $\otimes\mathcal{L}_\varphi$ and $G'$ is that of the tensor product, then $\langle G, Z\rangle = \langle G',Z\rangle$ as subgroups of $\GL_D(\overline{\mathbb{Q}_\ell})$, where $Z$ is the central subgroup of $\GL_D(\overline{\mathbb{Q}_\ell})$ of order equal to the order of $\varphi$. In particular, if $G$ is finite almost quasisimple, then $G'$ is also finite almost quasisimple with the same nonabelian composition factor. 

It is also impossible to put any additional restriction on $n$ and $q$ (recall that we do assume that $n\geq 3$). For all pairs $(q, n)$ the sheaves we obtain from \autoref{large_rank} and \autoref{small_rank} by setting $m=1$ and $\varphi=\mathds{1}$ were already studied by Katz and Tiep \cite{KT1}. All of them do have the desired geometric monodromy groups in irreducible Weil representations, cf. \cite[Corollary 8.2]{KT1}. Since we already have some restrictions on $m$, further reduction is not likely at this point. It turns out that this is indeed the case, as we will see in the next section.

\section{Computing the Geometric Monodromy Groups}

Although the sheaves in \autoref{large_rank} and \autoref{small_rank} survived our attempts to remove the non-examples of \eqref{hyp_condition}, we still need to show that these sheaves do have such geometric monodromy groups. The method we will use to compute these monodromy groups is based on the arguments in \cite{KT1}, which discusses the sheaves in \autoref{large_rank} and \autoref{small_rank} with $m=1$ and $\varphi=\mathds{1}$. 

The plan is as follows. We first form a directed sum of appropriately chosen sheaves from \autoref{large_rank} and \autoref{small_rank}. This direct sum will have a ``nice'' trace function, which becomes even nicer if we take a Kummer pullback. Then we can use results in \cite{KT1} to see that the monodromy representation must be the restriction of the total Weil representation of $\GL_n(\mathbb{F}_q)$ to some subgroup containing $\SL_n(\mathbb{F}_q)$. We use this fact to prove that the geometric monodromy group of the original direct sum before taking pullback is a quotient of $\GL_n(\mathbb{F}_q)$, and that the monodromy representation is a direct sum of certain irreducible Weil representations.

For the rest of this section, let $n, q, \chi, \psi$ be defined as in the previous section. Fix $b, c, m\in \mathbb{Z}$ and a multiplicative character $\varphi$ which satisfy the conditions in \autoref{large_rank}, namely
\begin{enumerate}[label={\upshape{(\roman*)}}]
\item $1\leq m\leq n-1$,
\item $bC - cB=1$ (so that $n$ and $m$, or equivalently $\frac{q^{n-m}-1}{q-1}$ and $\frac{q^m-1}{q-1}$, are coprime to each other), and
\item $\gcd(A, \frac{q-1}{\gcd(c,q-1)})=1$,
\end{enumerate}
where as in the proof of \autoref{small_rank}, we set $$A:=\frac{q^n-1}{q-1},\ B:=\frac{q^m-1}{q-1},\ C:=\frac{q^{n-m}-1}{q-1}.$$ Consider the following irreducible hypergeometric sheaves as in \autoref{large_rank}: $$\mathcal{H}_{j} := \mathcal{H}yp_\psi( \Char(A,\chi^{(b+c)j}); \Char(B,\chi^{bj})\cup \Char(C,\chi^{cj}))\text{ for }j=1,\dots,q-2$$ and one from \autoref{small_rank}: $$\mathcal{H}_0 := \mathcal{H}yp_\psi(\Char(A)\setminus\{\mathds{1}\}; \Char(B)\cup(\Char(C)\setminus\{\mathds{1}\})).$$

We first compute the trace functions of these sheaves. 
\begin{prop}\label{small_trace}
$\mathcal{H}_0$ is geometrically isomorphic to the lisse $\overline{\mathbb{Q}_\ell}$-sheaf $\mathcal{G}_0$ over $\mathbb{G}_m/\mathbb{F}_q$, which is pure of weight $0$ and whose trace function at each point $u\in K^\times$ of each finite extension $K/\mathbb{F}_q$ is given by 
\begin{align*}
u\in K^\times &\mapsto -1 + |\{v\in K^\times \mid u^{-b}v^B + u^{cq^m}v^{-Cq^m}=1\}|.
\end{align*}
\end{prop}

\begin{proof}
$\mathcal{H}_0$ is, by definition, the multiplicative $!$ convolution of three Kloosterman sheaves $\mathcal{K}l_\psi(\Char(A)\setminus\{\mathds{1}\})$, $\inv^*\mathcal{K}l_{\overline{\psi}}(\Char(B))$, and $\inv^*\mathcal{K}l_{\overline{\psi}}(\Char(C)\setminus\{\mathds{1}\})$. By \cite[Lemma 1.1 and 1.2]{KRLT3}, $\mathcal{H}_0$ is geometrically isomorphic to the lisse $\overline{\mathbb{Q}_\ell}$-sheaf over $\mathbb{G}_m$, pure of weight $4$, whose trace function at $u\in K^\times$ is given by the convolution
\begin{align*}
& \sum_{\substack{r,s,t\in K^\times\\rst=u}}\left(-\sum_{x\in K}\psi_K( Ax-\frac{1}{r}x^A)\right)\left(\sum_{\substack{y\in K\\ y^B=s^{-1}}}\psi_K(-By)\right)\left(-\sum_{z\in K}\psi_K(tz^C-Cz)\right)\\&= \sum_{x,z\in K,y,t\in K^\times}\psi_K( -\frac{t}{u}x^Ay^{-B}+tz^C+x-y-z)
\end{align*}
where the equality follows from the fact that $A\equiv B\equiv C\equiv 1$ mod $q$. Since $\psi$ is nontrivial and irreducible, we have $\sum_{x\in K}\psi_K(x)=0$. Using this, we can rewrite the above number as 
\begin{align*}
&\sum_{x,z\in K,y\in K^\times}\left(\psi_K(x-y-z)\sum_{t\in K^\times}\psi_K( t(-\frac{1}{u}x^Ay^{-B}+z^C))\right)\\=&\sum_{x,z\in K,y\in K^\times}\psi_K(x-y-z)\left(-1+\sum_{t\in K}\psi_K( t(-\frac{1}{u}x^Ay^{-B}+z^C))\right)\\=&-\sum_{z\in K,y\in K^\times}\left(\psi_K(-y-z)\sum_{x\in K}\psi_K(x)\right) + (\#K)\sum_{\substack{x,z\in K,y\in K^\times\\x^A=uy^Bz^C}}\psi_K(x-y-z)\\=&(\#K)(-1+\sum_{\substack{x,y,z\in K^\times\\x^A=uy^Bz^C}}\psi_K(x-y-z)).
\end{align*}
There is a bijection between the sets $\{(x,y,z)\mid x,y,z\in K^\times, x^A=uy^Bz^C\}$ and $K^\times \times K^\times$ given by 
\begin{align*}
(x,y,z) &\mapsto (x, x^{b+cq^m}y^{-b}z^{-c}) \\ (x, xu^cv^{-C}, x^{q^m}u^{-b}v^B) &\mapsfrom (x,v).
\end{align*}
By applying this change of variables, we can rewrite the above expression as
\begin{align*}
& (\#K)(-1 + \sum_{x,v\in K^\times}\psi_K(x - xu^cv^{-C}-x^{q^m}u^{-b}v^B)) \\=& (\#K)(-1 + \sum_{v\in K^\times}\sum_{x\in K^\times}\psi_K(x^{q^m}(1-u^{cq^m}v^{-Cq^m} - u^{-b}v^B))) \\=&(\#K)(-1 + \sum_{v\in K^\times}(-1) + \sum_{\substack{v\in K^\times\\1-u^{cq^m}v^{-Cq^m}-u^{-b}v^B=0}}(\#K)) \\=& (\#K)^2(-1 + |\{v\in K^\times \mid u^{-b}v^B + u^{cq^m}v^{-Cq^m} = 1\}|).
\end{align*}
Now apply a Tate twist $\mathbf{(2)}$ to obtain a lisse $\overline{\mathbb{Q}_\ell}$-sheaf $\mathcal{G}_{0}$ on $\mathbb{G}_m/\mathbb{F}_p$, pure of weight $0$ and geometrically isomorphic to $\mathcal{H}_0$, whose trace function is given by $$u\in K^\times \mapsto -1 + |\{v\in K^\times \mid u^{-b}v^{B}+u^{cq^m}v^{-Cq^m} = 1\}|.$$ 
\end{proof}

By a similar argument and \cite[5.6.2]{K-GKM}, we get the following:

\begin{prop}\label{big_trace}
$\mathcal{H}_{j}$ is geometrically isomorphic to the lisse $\overline{\mathbb{Q}_\ell}$-sheaf $\mathcal{G}_{j}$ on $\mathbb{G}_m/\mathbb{F}_q$, which is pure of weight $0$ and whose trace function at each point $u\in K^\times$ of each finite extension $K/\mathbb{F}_q$ is given by 
\begin{align*}
u\in K^\times \mapsto &\sum_{\substack{v\in K^\times\\u^{-b}v^{B}+u^{cq^m}v^{-Cq^m}=1}}\chi_K(v^j).
\end{align*}
\end{prop}

\begin{crl}
The sheaves $\mathcal{H}_{j}$, $0\leq j\leq q-2$, have finite geometric monodromy groups. 
\end{crl}

\begin{proof}
By \autoref{small_trace} and \autoref{big_trace}, they are geometrically isomorphic to the sheaves $\mathcal{G}_{j}$, which are pure of weight $0$ and whose trace functions have algebraic integer values. By \cite[Theorem 8.14.4]{K-ESDE}, their geometric monodromy groups are finite.
\end{proof}
 
Let $\mathcal{W}$ be the direct sum of the sheaves $\mathcal{G}_0$ in \autoref{small_trace} and $\mathcal{G}_j$ in \autoref{big_trace}: $$\mathcal{W} := \bigoplus_{j=0}^{q-2}\mathcal{G}_j.$$
\begin{prop}\label{direct_sum_trace}
$\mathcal{W}$ is geometrically isomorphic to the lisse $\overline{\mathbb{Q}_\ell}$-sheaf whose trace function is given by 
\begin{align*}
u\in K^\times &\mapsto -2 + |\{ w\in K \mid u^bw^{q^m}-u^{b+cq^m}w^{q^n}=w\}| .
\end{align*}
In particular, the values of this trace function is $-2$ plus either $1$ or a power of $q$.
\end{prop}

\begin{proof}
The trace function of $\mathcal{W}$ is given by the sum of trace functions of $\mathcal{G}_j$ computed in \autoref{small_trace} and \autoref{big_trace}:
\begin{align*}
u\in K^\times \mapsto -1 + \sum_{j=0}^{q-2}\sum_{\substack{v\in K^\times\\u^{-b}v^{B}+u^{cq^m}v^{-Cq^m}=1}}\chi_K(v^j).
\end{align*}
The sum $\sum_{j=0}^{q-2}\chi_K(v^j)$ is $q-1$ if $\Norm_{K/\mathbb{F}_q}(v)=1$, that is, if $v$ is a $q-1$th power in $K^\times$; it is $0$ otherwise. Also, if $v$ is a $q-1$th power in $K^\times$, then it has exactly $q-1$ distinct $q-1$th roots in $K^\times$, since $\mathbb{F}_q\subseteq  K$ has all $q-1$ distinct $q-1$th roots of unity. Therefore, the trace becomes 
\begin{align*}
& -1 + \sum_{\substack{v\in K^\times\\u^{cq^m}v^{-Cq^m}+u^{-b}v^{B}=1, \Norm_{K/\mathbb{F}_q}(v)=1}}(q-1) \\&= -1 + (q-1)|\{ v\in K^\times \mid u^{-b}v^B+u^{cq^m}v^{-Cq^m}=1,\ \Norm_{K/\mathbb{F}_q}(v)=1\}| \\&= -1 + |\{ w\in K^\times \mid u^{-b}w^{(q-1)B}+u^{cq^m}w^{-(q-1)Cq^m}=1\}|\\&= -1 + |\{ w\in K^\times \mid 1+u^{b+cq^m}w^{-(q-1)A}=u^bw^{-(q-1)B}\}|.
\end{align*}
By mapping $w$ to $w^{-1}$ we can write this as 
\begin{align*}
&-1 + |\{ w\in K^\times \mid u^bw^{(q-1)B}-u^{b+cq^m}w^{(q-1)A}=1\}|\\=&-2 + |\{ w\in K \mid u^bw^{q^m}-u^{b+cq^m}w^{q^n}=w\}|.
\end{align*}
Note that the set $\{ w\in K \mid u^bw^{q^m}-u^{b+cq^m}w^{q^n}=w\}$ forms an $\mathbb{F}_q$-vector subspace of $K$. Therefore, its size is either $1$ or a power of $q$. 
\end{proof}

To apply the results in \cite{KT1}, we take a Kummer pullback of $\mathcal{W}$.

\begin{prop}\label{pullback_trace}
The trace of $[\frac{(q-1)C}{\gcd(q-1, c)}]^*\mathcal{W}$ is $$u\in K^\times \mapsto -1 + |\{ w\in K^\times \mid w^{q^{m}-1}-w^{q^n-1}=u^{\frac{q-1}{\gcd(q-1,c)}}\}|.$$ This is $-2$ plus either $1$ or a power of $q$. Also, the trace of $[\frac{(q-1)C}{\gcd(q-1, c)}]^*\mathcal{G}_{0}$ is $$u\in K^\times \mapsto -1 +  |\{ w\in K^\times \mid w^{B}-w^{A}=u^{\frac{q-1}{\gcd(q-1,c)}}\}|.$$
\end{prop}

\begin{proof}
From \autoref{direct_sum_trace}, we can compute the trace function of $[\frac{(q-1)C}{\gcd(q-1,c)}]^*\mathcal{W}=[\frac{(q-1)C}{\gcd(q-1,c)}]^*\bigoplus_{j=0}^{q-2}\mathcal{G}_j$: 
\begin{align*}
u\in K^\times &\mapsto  -2 + |\{ w\in K \mid u^{\frac{(q-1)Cb}{\gcd(q-1,c)}}w^{q^{m}}-u^{\frac{(q-1)C(b+cq^m)}{\gcd(q-1,c)}}w^{q^n}=w\}|\\&=  -1 + |\{ w\in K^\times \mid u^{\frac{(q-1)Cb}{\gcd(q-1,c)}}w^{q^{m}-1}-u^{\frac{(q-1)C(b+cq^m)}{\gcd(q-1,c)}}w^{q^n-1}=1\}|.
\end{align*}
The map $w \mapsto wu^{-c/\gcd(q-1,c)}$ is a bijection from $K$ to itself, so we can rewrite the set in the above trace function as 
\begin{align*}
&\{ w\in K^\times \mid u^{\frac{(q-1)Cb}{\gcd(q-1,c)}}(wu^{-\frac{c}{\gcd(q-1,c)}})^{q^{m}-1}-u^{\frac{(q-1)C(b+cq^m)}{\gcd(q-1,c)}}(wu^{-\frac{c}{\gcd(q-1,c)}})^{q^n-1}=1\} \\=&\{ w\in K^\times \mid u^{\frac{(q-1)(Cb-Bc)}{\gcd(q-1,b)}}w^{q^{m}-1}-u^{\frac{(q-1)(Cb+Ccq^m-cA)}{\gcd(q-1,c)}}w^{q^n-1}=1\}\\=&\{ w\in K^\times \mid u^{-\frac{q-1}{\gcd(q-1,c)}}w^{q^{m}-1}-u^{-\frac{q-1}{\gcd(q-1,c)}}w^{q^n-1}=1\}\\=&\{ w\in K^\times \mid w^{q^{m}-1}-w^{q^n-1}=u^{\frac{q-1}{\gcd(q-1,c)}}\}.
\end{align*}
Again, this set together with $0$ form a $\mathbb{F}_q$-vector space, so its size is a power of $q$. The expression for the trace of $[\frac{(q-1)C}{\gcd(q-1,c)}]^*\mathcal{G}_{0}$ at $u\in K^\times$ can be obtained using a similar argument and \autoref{small_trace}.
\end{proof}

\begin{crl}\label{pullback_iso}
$\overline{\mathbb{Q}_\ell}\oplus \inv^*[\frac{(q-1)C}{\gcd(q-1, c)}]^*\mathcal{H}_{0}$ is geometrically isomorphic to $[\frac{q-1}{\gcd(q-1, c)}]^*f_*\overline{\mathbb{Q}_\ell}$, where $f(t)\in \mathbb{F}_q[t]$ is the polynomial $f(t) = t^B-t^A$. Also, $\overline{\mathbb{Q}_\ell}\oplus\inv^*[\frac{(q-1)C}{\gcd(q-1, c)}]^*\mathcal{W}$ is geometrically isomorphic to $[\frac{q-1}{\gcd(q-1, c)}]^*F_*\overline{\mathbb{Q}_\ell}$, where $F(t)\in \mathbb{F}_q[t]$ is the polynomial $F(t) = t^{(q-1)B}-t^{(q-1)A}$.
\end{crl}

Now we are ready to prove the main results of this section.

\begin{thm}\label{pullback_group}
The geometric monodromy group $G=G_{\geom}$ of $[\frac{(q-1)C}{\gcd(q-1, c)}]^*\mathcal{W}$ satisfies $\SL_n(\mathbb{F}_q)=G^{(\infty)} \lhd G\lhd \GL_n(\mathbb{F}_q)$. The monodromy representation of $\overline{\mathbb{Q}_\ell}\oplus [\frac{(q-1)C}{\gcd(q-1, c)}]^*\mathcal{W}$ as a representation of $G$ is the restriction of the permutation representation of $\GL_n(\mathbb{F}_q)$ acting on $\mathbb{F}_q^n\setminus \{0\}$, that is, $\bigoplus_{j=0}^{q-2} \Weil_j$.
\end{thm}

\begin{proof}
We mimic the proof of \cite[Theorem 8.1]{KT1}. To use \cite[Theorem 6.8]{KT1}, we check the conditions for this theorem. We start with the condition (a) of \cite[Theorem 6.8]{KT1}. By \autoref{pullback_iso} and \cite[Lemma 5.1]{KT1}, the geometric monodromy group $G=G_{\geom}$ of $[\frac{(q-1)C}{\gcd(q-1, c)}]^*\mathcal{W}$ can be embedded in $S_{q^n-1}$ in a way such that the monodromy representation of $\overline{\mathbb{Q}_\ell}\oplus[\frac{(q-1)C}{\gcd(q-1, c)}]^*\mathcal{W}$, viewed as a representation of $G$, is the restriction of the natural permutation representation of $S_{q^n-1}$. Also, $\mathcal{W}$ is geometrically isomorphic to $\bigoplus_{j=0}^{q-2} [\frac{(q-1)C}{\gcd(q-1, c)}]^*\mathcal{G}_{j}$. We need to show that each $[\frac{(q-1)C}{\gcd(q-1, c)}]^*\mathcal{G}_{j}$ is irreducible. 

For $j=0$, the monodromy representation of the Kummer pullback $[\frac{(q-1)C}{\gcd(q-1, c)}]^*\mathcal{G}_{0}$ is the restriction of the monodromy representation of $\mathcal{G}_{0}$ to a normal subgroup $H$ of $G$ such that $G/H$ is cyclic of order dividing $\frac{(q-1)C}{\gcd(q-1, c)}$. We also know that by \cite[Proposition 1.2]{KRLT2}, $\mathcal{H}_{0}$ is not geometrically induced. Clifford correspondence now shows that $[\frac{(q-1)C}{\gcd(q-1, c)}]^*\mathcal{G}_{0}$ is isotypic, that is, all of its irreducible constituents are isomorphic to each other. Now Gallagher's theorem \cite[Corollary 6.17]{Isaacs} together with the extendibility of invariant characters of normal subgroups with cyclic quotient \cite[Corollary 11.22]{Isaacs} shows that $[\frac{(q-1)C}{\gcd(q-1, c)}]^*\mathcal{G}_{0}$ is irreducible.

For $j\neq 0$, we know that the dimension of the monodromy representation of $\mathcal{G}_{j}$ is $A$, which is relatively prime to $\frac{(q-1)C}{\gcd(q-1, c)}$. On the other hand, $[\frac{(q-1)C}{\gcd(q-1, c)}]^*\mathcal{G}_{j}$ is the restriction of $\mathcal{G}_{j}$ to a normal subgroup of index dividing $\frac{(q-1)C}{\gcd(q-1, c)}$. By \cite[Corollary 11.29]{Isaacs}, $\frac{(q-1)C}{\gcd(q-1, c)}$ times the dimension of an irreducible constituent of this restriction must be divisible by $A$. Therefore, $[\frac{(q-1)C}{\gcd(q-1, c)}]^*\mathcal{G}_{j}$ is irreducible, and condition (a) of \cite[Theorem 6.8]{KT1} holds. 

By \autoref{pullback_iso} and \cite[Lemma 5.1]{KT1}, the geometric monodromy group of $[\frac{(q-1)C}{\gcd(q-1, c)}]^*\mathcal{G}_{0}$ can be embedded in $S_A$ in a way such that $[\frac{(q-1)C}{\gcd(q-1, c)}]^*\mathcal{G}_{0}$, as a representation of this group, is the restriction of the natural deleted permutation representation of $S_A$. Also, the image of $\gamma_0$ in this monodromy group has simple spectrum, and its order is the least common multiple of the orders of upstairs characters of $\mathcal{H}_{0}$, possibly divided by a divisor of $\frac{(q-1)C}{\gcd(q-1, c)}$. Since $\frac{(q-1)C}{\gcd(q-1, c)}$ and $A$ are relatively prime, the order of the image of $\gamma_0$ is $A$. We also know that the image of $P(\infty)$ in the geometric monodromy group is a $p$-subgroup of order at least $q^{n-1}$ by \cite[Proposition 4.10]{KT2}. Thus we have condition (b) of \cite[Theorem 6.8]{KT1}. 

\autoref{pullback_trace} shows that the trace plus $1$ is always a power of $q$, so condition (c) is also satisfied. \cite[Theorem 6.8]{KT1} now tells us that we must have $$\SL_n(\mathbb{F}_q)\cong G^{(\infty)} \lhd G\lhd \GL_n(\mathbb{F}_q),$$ and $\overline{\mathbb{Q}_\ell}\oplus[\frac{(q-1)C}{\gcd(q-1, c)}]^*\mathcal{W}$ as a representation of $G$ is the restriction of $\bigoplus_{j=0}^{q-2}\Weil_j$ to $G$. 
\end{proof}

\begin{thm}\label{Ggeom_of_sum}
In the situation of \autoref{pullback_group}, the geometric monodromy group $G=G_{\geom}$ of $\mathcal{W}$ is isomorphic to $\GL_n(\mathbb{F}_q)/\langle \alpha^d I\rangle$. 
\end{thm}

\begin{proof}
Let $H$ be the geometric monodromy group of $[\frac{(q-1)C}{\gcd(q-1,c)}]^*\mathcal{W}$, so that $H\unlhd \langle g_0, H\rangle=G$. By \autoref{pullback_group}, $H$ is the image under the representation $\Weil_0'\oplus \bigoplus_{j=1}^{q-2}\Weil_j$ of a subgroup of $\GL_n(\mathbb{F}_q)$ containing $\SL_n(\mathbb{F}_q)$. Let $L:=E(H)=\SL_n(\mathbb{F}_q)= [H, H]$ be the quasisimple layer of $H$. Also for each $j=0,\dots,q-2$, let $G_j$ be the geometric monodromy group of $\mathcal{G}_j$.\\

(1) We first prove that $\mathbf{Z}(G)=\mathbf{C}_G(L)$.\\

Since $\mathcal{W}$ is the direct sum of irreducible representations $\mathcal{G}_j$ and each of them restricts to an irreducible representation of $L$, we can view $G$ and $L$ as subgroups of $\GL_{A-1}(\overline{\mathbb{Q}_\ell})\oplus\GL_{A}(\overline{\mathbb{Q}_\ell})\oplus\cdots\oplus\GL_{A}(\overline{\mathbb{Q}_\ell})$. By Schur's Lemma, we get $$\mathbf{C}_{\GL_{q^n-2}(\overline{\mathbb{Q}_\ell})}(L)=\mathbf{C}_{\GL_{q^n-2}(\overline{\mathbb{Q}_\ell})}(G) = \mathbf{Z}(\GL_{A-1}(\overline{\mathbb{Q}_\ell}))\oplus\mathbf{Z}(\GL_{A}(\overline{\mathbb{Q}_\ell}))\oplus\cdots\oplus\mathbf{Z}(\GL_{A}(\overline{\mathbb{Q}_\ell})).$$ Therefore $$\mathbf{C}_G(L) = \mathbf{C}_{\GL_{q^n-2}(\overline{\mathbb{Q}_\ell})}(L)\cap G =  \mathbf{C}_{\GL_{q^n-2}(\overline{\mathbb{Q}_\ell})}(G)\cap G = \mathbf{Z}(G).$$\\

(2) Next, we show that $G/\mathbf{Z}(G)\cong \PGL_n(\mathbb{F}_q)$. \\

Since $L$ is normal in $G$, the restriction of the monodromy representation of each $\mathcal{G}_j$ to $L$ is invariant under conjugation by elements of $G$. We already saw that they are precisely $\Weil_0'$ and $\Weil_j$, $j=1,\dots,q-2$. The only automorphisms of $\SL_n(\mathbb{F}_q)$ fixing each of these representations are the inner and diagonal automorphisms. Therefore we get $$G/\mathbf{Z}(G)=G/\mathbf{C}_G(L)\leq \PGL_n(\mathbb{F}_q).$$ 

On the other hand, the geometric monodromy group $G_j$ is the image of $G$ under the projection from $\GL_{A-1}(\overline{\mathbb{Q}_\ell})\oplus \GL_A(\overline{\mathbb{Q}_\ell})\oplus \cdots \oplus \GL_A(\overline{\mathbb{Q}_\ell})$ onto the $j$th summand. Therefore, $G_j$ is a quotient of $G$, and $G_j$ contains the image of $\SL_n(\mathbb{F}_q)$ acting on an irreducible Weil representation. In particular, $G_j$ is finite, almost quasisimple with unique nonabelian composition factor $\PSL_n(\mathbb{F}_q)$, so by \autoref{basic_facts}(d), $G_j/\mathbf{Z}(G_j)\cong \PGL_n(\mathbb{F}_q)$, except for the exceptional pairs $(n,q)=(3,2),(3,3), (3,4)$ of \autoref{basic_facts}(d). For $(n,q)=(3,2)$ and $(3,3)$, we have $\PGL_n(\mathbb{F}_q) = \PSL_n(\mathbb{F}_q)$, so we still get $G_j/\mathbf{Z}(G_j)\cong \PGL_n(\mathbb{F}_q)$. Since we already know that $G/\mathbf{Z}(G)\leq \PGL_n(\mathbb{F}_q)$, we get $G/\mathbf{Z}(G)\cong \PGL_n(\mathbb{F}_q)$.

For $(n,q)=(3,4)$, $A=21$ is divisible by $|\mathbb{F}_q^\times|=q-1=3$, so by \autoref{m2sp_elts}, no $p'$-element (in fact, no element at all) of $\SL_n(\mathbb{F}_q)$ has simple spectrum, which shows that $G_j/\mathbf{Z}(G_j)\neq \PSL_3(\mathbb{F}_4)$. Therefore, $G/\mathbf{Z}(G)$ is also not equal to $\PSL_3(\mathbb{F}_q)$. Since $\PGL_3(\mathbb{F}_4)/\PSL_3(\mathbb{F}_4)$ is simple and $G/\mathbf{Z}(G)\leq \PGL_3(\mathbb{F}_4)$, the equality must hold. \\

(3) We compute the order of $\mathbf{Z}(G)$. \\

Let $z\in \mathbf{Z}(G)$. By (1), $$z=(\epsilon_0 I, \epsilon_1 I, \dots, \epsilon_{q-2}I)\in \GL_{A-1}(\overline{\mathbb{Q}_\ell})\oplus\GL_{A}(\overline{\mathbb{Q}_\ell})\oplus\cdots\oplus\GL_{A}(\overline{\mathbb{Q}_\ell})$$ for some roots of unity $\epsilon_j\in \overline{\mathbb{Q}_\ell}^\times$, where $I$ denotes the identity matrix of appropriate size. Thus for each $j$, the trace of the action of $z$ on $\mathcal{G}_j$ is just $(A-\delta_{0,j})\epsilon_j$. 

Let $K$ be a finite extension of $\mathbb{F}_q$ on which $\mathcal{W}$ is defined and such that $[K:\mathbb{F}_q]$ is even, so that $\Norm_{K/\mathbb{F}_q}(-1)=1$. Since $\mathcal{G}_j$ is pure of weight $0$ and its geometric monodromy group $G_j$ is finite, it follows that the arithmetic monodromy group $G_{j,K}$ over $K$ of $\mathcal{G}_j$ is also finite, and hence the arithmetic monodromy group $G_K$ of $\mathcal{W}$ is also finite. By Chebotarev density theorem, every element of $G_K$ comes from a Frobenius in the arithmetic fundamental group. In particular, $z$ is the image of the Frobenius at some $u\in K^\times$.

The trace of the action of $z$ on $\mathcal{G}_j$, which is $(A-\delta_{0,j})\epsilon_j$, is also the value at $u$ of the trace functions we computed in \autoref{small_trace} and \autoref{big_trace}: $$(A-\delta_{0,j})\epsilon_j = -\delta_{0,j} + \sum_{\substack{v\in K^\times\\ u^{-b}v^A - v^{Cq^m} + u^{cq^m}=0}}\chi_K(v^j).$$ Note that there are at most $A$ elements $v\in K^\times$ satisfying $u^{-b}v^A-v^{Cq^m}+u^{cq^m}=0$, and each $\chi_K(v^j)$ is a root of unity. Therefore, the above equality forces that there are precisely $A$ such $v$ in $K^\times$ and that $\chi_K(v)^j = \epsilon_j$ for all such $v$. In particular, $\epsilon_1$ is a $(q-1)$th root of unity and $\epsilon_j = \epsilon_1^j$ for all $j$. Also, all such $v$ has the same $\Norm_{K/\mathbb{F}_q}(v)$, which is precisely the inverse image $\nu\in \mathbb{F}_q^\times$ of $\epsilon_1$ under $\chi$. 

Since $v$ are the roots of the polynomial $u^{-b}v^A-v^{Cq^m}+u^{cq^m}$, the product of all $v$ is exactly $(-1)^Au^{b+cq^m}$. Hence 
\begin{align*}
\nu^n&=\nu^A =\prod_{\substack{v\in K^\times\\u^{-b}v^{A}-v^{Cq^m}+u^{cq^m}=0}}\Norm_{K/\mathbb{F}_q}(v) =\Norm_{K/\mathbb{F}_q}(\prod_{v}(v)) \\&=\Norm_{K/\mathbb{F}_q}(  (-1)^Au^{b+cq^m})=\Norm_{K/\mathbb{F}_q}(u)^{b+c}
\end{align*}
where the last equality follows from $\Norm_{K/\mathbb{F}_q}(-1)=1$.

Let $d = \frac{q-1}{\gcd(q-1,b+c)}$, which is the order of $b+c\in \mathbb{Z}/(q-1)\mathbb{Z}$. Since $\nu\in \mathbb{F}_q^\times$ has order dividing $q-1$, we have
\begin{align*}
\nu^d = \nu^{d(bC-cB)} = \nu^{d(b(n-m)-cm)} = \nu^{bdn - (b+c)dm} = \nu^{bdn} = \Norm_{K/\mathbb{F}_q}(u)^{bd(b+c)}=1.
\end{align*}
Therefore $\epsilon_1=\chi(\nu)$ has order dividing $d$, so it must be a power of $\lambda^{b+c}$. Since this holds for all $z\in \mathbf{Z}(G)$, we get
$$\mathbf{Z}(G)\leq \langle (I, \lambda^{b+c}I, \lambda^{2(b+c)}I, \dots, \lambda^{(q-2)(b+c)}I)\rangle \cong \mathbb{Z}/d\mathbb{Z}.$$

Recall that the geometric determinant of $\mathcal{H}_{1}$ is $\mathcal{L}_{\chi^{b+c}}$ if $A$ is odd, and $\mathcal{L}_{\chi^{b+c}\chi_2}$ if $A$ is even. Also note that if $A$ is even, then $d$ is also even: we know that $bC-cB=1$, so that $bn-(b+c)m\equiv 1$ mod $q-1$. Since $A$ is even if and only if $n$ is even and $q$ is odd, it follows that $(b+c)m$ is odd, whence $d$ is even. Therefore, if $A$ is even, then we can choose a $j_0$ such that $\chi^{(b+c)j_0}=\chi_2$. Then the geometric determinant of $\mathcal{H}_{1+j_0}$ is $\mathcal{L}_{\chi^{b+c}}$. In particular, $G$ has $\mathbb{Z}/d\mathbb{Z}$ as a quotient. Since $L$ is perfect, this quotient map factors through $G/L$. Since $G/\mathbf{Z}(G)\cong \PGL_n(\mathbb{F}_q)$ and $|L|=|\SL_n(\mathbb{F}_q)|=|\PGL_n(\mathbb{F}_q)|$, we have $|\mathbf{Z}(G)|= |G/L|$. Therefore, $|\mathbf{Z}(G)|$ is divisible by $d$. Together with the above observations, this implies $\mathbf{Z}(G) \cong \mathbb{Z}/d\mathbb{Z}$ and $G/L \cong \mathbb{Z}/d\mathbb{Z}$.\\

(4) Now we prove that $G\cong \GL_n(\mathbb{F}_q)/\langle \alpha^d I \rangle$.\\

To prove this, we will find a surjective group homomorphism $F:\GL_n(\mathbb{F}_q) \rightarrow G$ with kernel $\langle \alpha^d I \rangle$. Since we already embedded these groups into $\GL_{q^n-2}(\overline{\mathbb{Q}_\ell})$ in a way such that $\SL_n(\mathbb{F}_q) = L$, we only need to extend this to $\GL_n(\mathbb{F}_q) = \langle \diag(\alpha,1,\dots, 1)\rangle \SL_n(\mathbb{F}_q)$. 

As we saw above, $\mathcal{L}_{\chi^{b+c}}$ is the geometric determinant of some $\mathcal{G}_j$. We can view this as a one-dimensional representation of $G$. Since $L\leq [G,G]$, $L$ lies in the kernel of this representation. Also, the image of $g_0$ under this representation has order equal to the order of this representation, which is $d=|G/L|$. Therefore $G=\langle g_0,L\rangle$. 

Since $g_0L$ generates $G/L$, $g_0\mathbf{Z}(G)L$ generates $G/\mathbf{Z}(G)L$. Note that \begin{align*}
G/\mathbf{Z}(G)L &\cong (G/\mathbf{Z}(G))/(\mathbf{Z}(G)L/\mathbf{Z}(G))\cong \PGL_n(\mathbb{F}_q)/\PSL_n(\mathbb{F}_q) \cong \GL_n(\mathbb{F}_q) / (\mathbf{Z}(\GL_n(\mathbb{F}_q))\SL_n(\mathbb{F}_q)).
\end{align*} 
Each generator of this quotient group is of the form $$\diag(\alpha^{t},1,\dots,1)\mathbf{Z}(\GL_n(\mathbb{F}_q))\SL_n(\mathbb{F}_q)$$ for some $t\in \mathbb{Z}$ relatively prime to $q-1$. Therefore, we can choose an integer $t_0$ relatively prime to $q-1$ and elements $z_0\in \mathbf{Z}(\GL_n(\mathbb{F}_q))$, $s_0\in \SL_n(\mathbb{F}_q)$ such that $$g_0 = z_0h_0s_0, \text{ where }h_0=\diag(\alpha^{t_0},1,\dots, 1).$$ 

Let $F:\GL_n(\mathbb{F}_q) \rightarrow  G$ be the map defined as $$ F(h_0^t s) = (g_0s_0^{-1})^ts= z_0^th_0^ts\text{ for each }t\in \mathbb{Z}\text{ and }s\in \SL_n(\mathbb{F}_q).$$ I claim that this map has the desired properties. First, we check the well-definedness: if $t_1, t_2\in \mathbb{Z}$ and $s_1, s_2\in \SL_n(\mathbb{F}_q)$ are such that $h_0^{t_1}s_1 = h_0^{t_2}s_2$, then $\diag(\alpha^{t_0(t_1-t_2)},1,\dots, 1) = h_0^{t_1-t_2} = s_2s_1^{-1} \in \SL_n(\mathbb{F}_q)$, so $t_1-t_2$ must be divisible by $q-1$. Then $z_0^{t_1-t_2}=1$, so
\begin{align*}
F(h_0^{t_1}s_1) &= z_0^{t_1} h_0^{t_1} s_1 =z_0^{t_2-t_1} z_0^{t_1}h_0^{t_2}s_2 = z_0^{t_2}h_0^{t_2}s_2 =F(h_0^{t_2}s_2).
\end{align*}
Therefore, $F$ is well-defined. We next check that it is a group homomorphism: for $t_1,t_2\in \mathbb{Z}$ and $s_1,s_2\in \SL_n(\mathbb{F}_q)$, we have
\begin{align*}
F(h_0^{t_1}s_1h_0^{t_2}s_2) &= F(h_0^{t_1+t_2}(h_0^{-t_2}s_1h_0^{t_2}s_2)) = z_0^{t_1+t_2}h_0^{t_1+t_2}(h_0^{-t_2}s_1h_0^{t_2}s_2) \\&= z_0^{t_1}z_0^{t_2}h_0^{t_1}s_1h_0^{t_2}s_2= z_0^{t_1}h_0^{t_1}s_1z_0^{t_2}h_0^{t_2}s_2= F(h_0^{t_1}s_1)F(h_0^{t_2}s_2).
\end{align*}
Since $G = \langle g_0, L\rangle = \langle g_0s_0^{-1}, L\rangle$, every element of $G$ can be written in the form $(g_0s_0^{-1})^t s$ for some $t\in \mathbb{Z}$ and $s\in L=\SL_n(\mathbb{F}_q)$. Therefore $F$ is surjective. 

Finally, we check that $\ker F=\langle \alpha^{d}I\rangle$. If $h_0^ts\in \ker F$ for $t\in \mathbb{Z}$ and $s\in \SL_n(\mathbb{F}_q)$, then $$ (g_0s_0^{-1})^t= (g_0s_0^{-1})^t s s^{-1}  = F(h_0^t s)s^{-1} = s^{-1}\in L.$$ Since $G=\langle g_0s_0^{-1}, L\rangle$, the element $(g_0s_0^{-1})L\in G/L$ has order exactly $d$. Hence $d$ must divide $t$, so we may write $t=dt'$ for some $t'\in \mathbb{Z}$. Then 
\begin{align*}
h_0^t s &= z_0^{-t}(g_0s_0^{-1})^{t} s = z_0^{-dt'}F(h_0^t s) = z_0^{-dt'}\in \mathbf{Z}(\GL_n(\mathbb{F}_q))=\langle \alpha I\rangle.
\end{align*}
Since $z_0^d$ has order dividing $\frac{q-1}{d}$, it follows that $z_0^{-dt'}\in \langle \alpha^{d}I\rangle$, so that $\ker F \leq \langle \alpha^d I\rangle$. Also, $$|\ker F| = \frac{|\GL_n(\mathbb{F}_q)|}{|G|} = \frac{(q-1)|\SL_n(\mathbb{F}_q)|}{d|L|} = \frac{q-1}{d} = |\langle \alpha^d I\rangle|.$$ Therefore $\ker F=\langle \alpha^d I\rangle$, so $\GL_n(\mathbb{F}_q)/\langle \alpha^d I\rangle \cong  G$.
\end{proof}

To determine the geometric monodromy groups of the summands $\mathcal{G}_j$, we need to know what are the irreducible constituents of the monodromy representation of $\mathcal{W}$ as a representation of $\GL_n(\mathbb{F}_q)$. Motivated by \cite[Theorem 16.6]{KT3} for $\SU_n(\mathbb{F}_q)$ and its extension \cite[Theorem 1.2]{TY} to $\GL_n(\mathbb{F}_q)$, we will prove an analogous result for $\GL_n(\mathbb{F}_q)$ and use it to study this representation. By a slight abuse of notation, we will also denote by $\Weil$ and $\Weil_j$ the representations of $\GL_n(\mathbb{F}_q)$ corresponding to these modules. 

\begin{thm}\label{twisted_Weil}
 Suppose that a complex representation $\Phi$ of $\GL_n(\mathbb{F}_q)$ has the following properties:
\begin{enumerate}[label={\upshape{(\alph*)}}]
\item $\Weil_0 = \Weil_0'\oplus \mathds{1}$ is isomorphic to a subrepresentation of $\Phi$,
\item $\Phi|_{\SL_n(\mathbb{F}_q)} \oplus \mathds{1} \cong \Weil|_{\SL_n(\mathbb{F}_q)}$, and
\item The values of the character afforded by $\Phi\oplus \mathds{1}$ are in $\{1,q,q^2,\dots, q^n\}$.
\end{enumerate}
Then $\Phi = \bigoplus_{j=0}^{q-2} \Weil_j\otimes X^{ej}$ for some integer $e$, where $X:= \GL_n(\mathbb{F}_q) \xrightarrow{\det} \mathbb{F}_q^\times \xrightarrow{\alpha\mapsto \lambda} \mathbb{C}^\times$ is a one-dimensional representation of order $q-1$. 
\end{thm}

\begin{proof}
For each integer $b', c'$ such that $b'\not\equiv c'$ mod $\frac{q^{n-1}-1}{q-1}$, let $$x_{b', c'} :=\begin{pmatrix}\alpha^{b'} & 0\\0 & \alpha_{n-1}^{c'-b'}\end{pmatrix}\in \GL_1(\mathbb{F}_q)\oplus \GL_{n-1}(\mathbb{F}_q) < \GL_n(\mathbb{F}_q).$$ If $c'-b'$ is divisible by $\frac{q^{n-1}-1}{q-1}$, then instead define $$x_{b', c'} :=\begin{pmatrix}\alpha^{b'} & 0\\0 & \alpha_{n-1}^{c'-b'+q-1}\end{pmatrix}\in \GL_1(\mathbb{F}_q)\oplus \GL_{n-1}(\mathbb{F}_q) < \GL_n(\mathbb{F}_q).$$ Since $n\geq 3$, we have $\frac{q^{n-1}-1}{q-1}\geq \frac{q^2-1}{q-1}=q+1>q-1$, so if $c'-b'$ is divisible by $\frac{q^{n-1}-1}{q-1}$ then $c'-b'+q-1$ is not. Recall that $\alpha_{n-1}$ permutes all $q^{n-1}-1$ nonzero vectors of $\mathbb{F}_q^{n-1}$ cyclically, and that $\alpha_{n-1}^{\frac{q^{n-1}-1}{q-1}}=\alpha I$ acts on $\mathbb{F}_q^{n-1}$ as a scalar. Therefore, the lower diagonal block of $x_{b',c'}$ cannot have any eigenvector in $\mathbb{F}_q^{n-1}$. It follows that the only eigenvectors of $x_{b',c'}$ in $\mathbb{F}_q^n$ are those of the form $(u, 0)\in \mathbb{F}_q\oplus \mathbb{F}_q^{n-1}$ for nonzero $u\in \mathbb{F}_q$, and they have eigenvalue $\alpha^{b'}$. 

\autoref{equiv_relation} tells us that the spectrum of the action of $x_{b',c'}$ on $\Weil_j$ can be partitioned into subsets, and each of this subsets is, in the notations of section $2$, the set of all $s_v$th roots of $\lambda^{t_vj}$, where $v$ is a representative of the $\sim_{x_{b',c'}}$-equivalence class corresponding to this subset. The sum of the eigenvalues in a such subset is $0$ unless $s_v=1$, in which case the sum is simply $\lambda^{t_v j}$. But $s_v=1$ means that $v$ is an eigenvectors of $x_{b',c'}$, and as we saw above, the only eigenvectors of $x_{b',c'}$ are those of the form $(u,0)\in \mathbb{F}_q\oplus \mathbb{F}_q^{n-1}$ with $u\neq 0$. Therefore, the trace of the action of $x_{b',c'}$ on $\Weil_j$ is $\lambda^{b'j}$.

The determinant of $x_{b',c'}$ is $$\det x_{b',c'} = \alpha^{b'} N_{\mathbb{F}_{q^{n-1}}/\mathbb{F}_q}(\alpha_{n-1}^{c'-b'}) =  \alpha^{b'} (\alpha_{n-1}^{c'-b'})^{\frac{q^{n-1}-1}{q-1}} =\alpha^{b' + (c'-b')} = \alpha^{c'}$$ if $c'\not\equiv b'$ mod $\frac{q^{n-1}-1}{q-1}$, and for the other cases we also get $$\det x_{b',c'} = \alpha^{b'} N_{\mathbb{F}_{q^{n-1}}/\mathbb{F}_q}(\alpha_{n-1}^{c'-b'+q-1})  =\alpha^{b' + (c'-b' + q -1)} = \alpha^{c'}.$$ Thus we get $X(x_{b',c'}) =\lambda^{c'}$.

Since $\Phi|_{\SL_n(\mathbb{F}_q)} \oplus \mathds{1} \cong \Weil|_{\SL_n(\mathbb{F}_q)}$, and $\Phi$ has a subrepresentation isomorphic to $\Weil_0$, we must have $$\Phi \cong \Weil_0 \oplus \bigoplus_{j=1}^{q-2} \Weil_j\otimes X^{i_j}=\bigoplus_{j=0}^{q-2} \Weil_j\otimes X^{i_j}$$ for some integers $i_j$ with $i_0=0$. Then by the assumptions and the above calculations, for each pair of integers $b', c'$ we get
\begin{align*}
1 + \Trace \Phi(x_{b',c'}) =1 + \sum_{j=0}^{q-2} \lambda^{b'j + c'i_j} \in \{1,q,\dots, q^n\}.
\end{align*}
This is a sum of $q$ roots of unity, so this is actually in $\{1, q\}$. 

For each $b',c'\in \mathbb{Z}$ and $j\in \{0,\dots,q-2\}$, let $d_{b',c',j}$ be the unique integer in $\{0,\dots, q-2\}$ such that $d_{b',c',j} \equiv b'j + c'i_j$ mod $q-1$. Consider the polynomial $$P_{b', c'}(T) := \sum_{j=0}^{q-2}T^{d_{b',c',j}} \in \mathbb{Z}[T].$$ Then every coefficient is nonnegative, and the constant term is positive since $d_{b',c',0} = 0$. Clearly $P_{b',c'}(1) = q-1$, and for each integer $r$, we have $P_{b',c'}(\lambda^r)\in \{0,q-1\}$. By \cite[Theorem 3.1 and Lemma 3.2]{TY}, there exists $e\in \mathbb{Z}$ such that $i_j \equiv ej$ mod $q-1$ for all $j=0,\dots,q-2$. Since $X$ has order $q-1$, we get $$\Phi \cong \bigoplus_{j=0}^{q-2} \Weil_j\otimes X^{ej}.$$

\end{proof}

\begin{crl}\label{Ggeom_of_constituents}
The monodromy representation of $\mathcal{G}_0$ as a representation of $\GL_n(\mathbb{F}_q)$ is $\Weil_0'$. The set of monodromy representations of $\mathcal{G}_j$ for $j=1,\dots,q-2$ is equal to the set $\{\Weil_j \otimes X^{ej}\mid j=1,\dots, q-2\}$, where $e$ is an integer such that $ne+1$ has order exactly $d$ in $\mathbb{Z}/(q-1)\mathbb{Z}$, and $X$ is as in \autoref{twisted_Weil}. 
\end{crl}

\begin{proof}
We know that the restriction of the monodromy representation of $\mathcal{W}$ to $L$ is the same as the  restriction of $\Weil_0'\oplus\bigoplus_{j=1}^{q-2}\Weil_j$ to $\SL_n(\mathbb{F}_q)$. The only irreducible constituent of rank $A-1$ of them are $\mathcal{G}_0$ and $\Weil_0'$. Therefore, $\mathcal{G}_0$ as a representation of $\GL_n(\mathbb{F}_q)$ is $\Weil_0'\otimes Y$ for some one-dimensional representation $Y$ of $\GL_n(\mathbb{F}_q)$ which is trivial on $\langle \alpha^d I\rangle$. Since $g_0$ has simple spectrum on $\mathcal{G}_0$, it also has simple spectrum on $\Weil_0'$. By \autoref{m2sp_elts}, $g_0$ as an element of $\GL_n(\mathbb{F}_q)/\langle \alpha^d I\rangle$ is the image of $\alpha_n^a\in \GL_n(\mathbb{F}_q)$ for some integer $a$ relatively prime to $A$. The spectrum of the action of $\alpha_n^a$ on $\Weil_0'$ is the set of $A$th roots of unity other than $1$, which is exactly the spectrum of $g_0$ on $\mathcal{G}_0$. Therefore, as a representation of $\GL_n(\mathbb{F}_q)/\langle \alpha^d I\rangle$, $Y$ contains both $g_0$ and $L$ in the kernel. Since $G=\langle g_0, L\rangle$, $Y$ is trivial. Therefore, $\overline{\mathbb{Q}_\ell}\oplus \mathcal{W}$ as a representation of $\GL_n(\mathbb{F}_q)$ satisfies the condition of \autoref{twisted_Weil}. By \autoref{twisted_Weil}, $\mathcal{W}$ must be $\Weil_0'\oplus \bigoplus_{j=1}^{q-2} \Weil_j\otimes X^{ej}$ for some integer $e$, where $X$ is as in \autoref{twisted_Weil}. 

The restriction of $\Weil_j\otimes X^{ej}$ to $\mathbf{Z}(\GL_n(\mathbb{F}_q)) = \langle \alpha I\rangle$ maps $\alpha I$ to $\lambda^{j + nej}=\lambda^{(ne+1)j}$. Since the kernel of the monodromy representation of $\mathcal{W}$ is $\langle \alpha^d I \rangle$, it follows that $\lambda^{(ne+1)j}$ has order dividing $d$ for all $j$, and order exactly $d$ for at least one $j$. Therefore $ne+1$ has order exactly $d$ in $\mathbb{Z}/(q-1)\mathbb{Z}$.
\end{proof}

We finish this section with the following result on the geometric monodromy group of a Kummer pullback of $\mathcal{W}$. The proof is entirely analogous to \cite[Corollary 8.4]{KT1}.

\begin{crl}\label{Ggeom_pullback_by_d}
Let $f$ be a divisor of $d$. The geometric monodromy group $G_f$ of the Kummer pullback $[f]^*\mathcal{W}$ is $(\SL_n(\mathbb{F}_q)\rtimes \langle \diag(\alpha^f, 1, \dots, 1)\rangle)/\langle \alpha^d I\rangle $.
\end{crl}

\section{Abhyankar's Theorem on Galois Groups of Trinomials}

In \cite[Section 9]{KT1}, Katz and Tiep related their hypergeometric sheaves to Abhyankar's result \cite[Theorem 1.2]{Ab94} on the Galois groups of certain polynomials. Since those sheaves are precisely the sheaf $\mathcal{W}$ in the previous section with $m=n-1$, $b=1$ and $c=0$, it is natural to ask if this connection can be generalized to other values of $m$.

Consider the polynomial described in \cite[Theorem 1.2]{Ab94}: $$F(T, U) := T^{q^n-1} - xU^rT^{q^m-1} + yU^s\in \overline{\mathbb{F}_q}(U)[T]$$ where $n, m$ are integers relatively prime to each other, $x, y$ are nonzero elements in $\overline{\mathbb{F}_q}$, and $r, s$ are nonnegative integers such that $r(q^n-1)\neq s(q^n-q^m)$. Let $y', z\in \overline{\mathbb{F}_q}$ be numbers such that $(y')^{q^n-1}=y$, $z^{r(q^n-1)-s(q^n-q^m)} = x^{-1}(y')^{q^n-q^m}$ and let $K_0$ be a finite extension of $\mathbb{F}_q$ such that $x, y', z\in K_0$. The Galois group of this polynomial over $\overline{\mathbb{F}_q}(U)$ is the geometric monodromy group $G_{\geom, \mathcal{A}}$ of the lisse $\overline{\mathbb{Q}_\ell}$-sheaf $\mathcal{A}$ over $\mathbb{G}_m/K_0$ whose trace at $u\in K^\times$ for a finite extension $K$ of $K_0$ is the number of solutions of the equation $$F(T,u)=T^{q^{n}-1} - xu^rT^{q^m-1}+yu^s =0.$$ If we take the $[q^n-1]^*$ Kummer pullback, then we get a lisse sheaf over $\mathbb{G}_m/K_0$ whose trace at $u\in K^\times$ is 
\begin{align*}
&|\{w \in K^\times \mid w^{q^n-1} - xu^{(q^n-1)r}w^{q^m-1} + yu^{(q^n-1)s} = 0\}|\\=&|\{w \in K^\times \mid (b'u^sw)^{q^n-1} - xu^{(q^n-1)r}(b'u^sw)^{q^m-1} + yu^{(q^n-1)s} = 0\}|\\=&|\{w \in K^\times \mid w^{q^n-1} - x(y')^{-q^n+q^m}u^{r(q^n-1)-s(q^n-q^m)}w^{q^m-1} + 1 = 0\}|
\end{align*}
The geometric monodromy group $G_{\geom, [q^n-1]^*\mathcal{A}}$ of this sheaf satisfies $$G_{\geom, [q^n-1]^*\mathcal{A}} \unlhd G_{\geom, \mathcal{A}},  G_{\geom, \mathcal{A}} /G_{\geom, [q^n-1]^*\mathcal{A}}\text{ is cyclic of order dividing }q^n-1.$$ 

Let $\mathcal{B}$ be the lisse sheaf over $K_0$ obtained by taking multiplicative translate $[u\mapsto zu]^*[q^n-1]^*\mathcal{A}$. The geometric monodromy group $G_{\geom, \mathcal{B}}$ is the same as $G_{\geom,[q^n-1]^*\mathcal{A}}$, and the trace of $\mathcal{B}$ at $u\in K^\times$ is 
\begin{align*}
&|\{w \in K^\times \mid w^{q^n-1} - x(y')^{-q^n+q^m}(zu)^{r(q^n-1)-s(q^n-q^m)}w^{q^m-1} + 1 = 0\}|\\=&|\{w \in K^\times \mid w^{q^n-1} - u^{r(q^n-1)-s(q^n-q^m)}w^{q^m-1} + 1 = 0\}|.
\end{align*}
Since $r(q^n-1)-s(q^n-q^m)$ is a nonzero multiple of $q-1$, $\mathcal{B}$ is geometrically isomorphic to the $[(r(q^n-1)-s(q^n-q^m))/(q-1)]^*$ Kummer pullback of the lisse sheaf $\mathcal{C}$ over $\mathbb{G}_m/K_0$ whose trace at $u\in K^\times$ is 
\begin{align*}
&|\{w \in K^\times \mid w^{q^n-1} - u^{q-1}w^{q^m-1} + 1 = 0\}|.
\end{align*}
On the other hand, if we choose integer $b,c$ such that $n, m, b, c$ satisfies the conditions of \autoref{large_rank}, and use the notations from the previous section, then the $[q^m]^*$ Frobenius pullback of $\mathcal{C}$ has trace 
\begin{align*}
&|\{w \in K^\times \mid w^{q^n-1} - u^{(q-1)q^m}w^{q^m-1} + 1 = 0\}|\\=&|\{w \in K^\times \mid w^{q^n-1} - u^{b(q^n-1)-(b+cq^m)(q^m-1)}w^{q^m-1} + 1 = 0\}|\\=&|\{w \in K^\times \mid (u^{b+cq^m}w)^{q^n-1} - u^{b(q^n-1)-(b+cq^m)(q^m-1)}(u^{b+cq^m}w)^{q^m-1} + 1 = 0\}|\\=&|\{w \in K^\times \mid u^{(b+cq^m)(q^n-1)}w^{q^n-1} - u^{b(q^n-1)}w^{q^m-1} + 1 = 0\}|.
\end{align*}
Therefore, by \autoref{direct_sum_trace}, $[q^m]^*\mathcal{C}$ is geometrically isomorphic to $[q^n-1]^*(\overline{\mathbb{Q}_\ell}\oplus\mathcal{W})$. 

Using these geometric isomorphisms, we can show that the geometric monodromy group of $\mathcal{C}$ is isomorphic to that of $[q-1]^*\overline{\mathbb{Q}_\ell}\oplus \mathcal{W}$, which is $\SL_n(\mathbb{F}_q)$ by \autoref{Ggeom_pullback_by_d}. Hence, $G_{\geom, \mathcal{B}}$ are also isomorphic to $\SL_n(\mathbb{F}_q)$, since $\mathcal{B}$ is geometrically isomorphic to $[(r(q^n-1)-s(q^n-q^m))/(q-1)]^*\mathcal{C}$ as we saw above. Since $G_{\geom, [q^n-1]^*\mathcal{A}}$ is isomorphic to $G_{\geom, \mathcal{B}}$, it follows that $G_{\geom, \mathcal{A}}$ contains a normal subgroup $\SL_n(\mathbb{F}_q)$, and $G_{\geom, \mathcal{A}}/\SL_n(\mathbb{F}_q)$ is cyclic of order dividing $q^n-1$. On the other hand, the $q^n-1$ roots of $F(T,U)\in \overline{\mathbb{F}_q}(U)[T]$ together with $0$ form an $n$-dimensional $\mathbb{F}_q$-vector space. Each element of the Galois group acts $\mathbb{F}_q$-linearly on this set, so the Galois group $G_{\geom,\mathcal{A}}$ is contained in $\GL_n(\mathbb{F}_q)$. Thus we get the conclusion of case (3) of \cite[Theorem 1.2]{Ab94}.

\end{document}